\documentclass[english]{article}
\usepackage[T1]{fontenc}
\usepackage[latin9]{inputenc}
\usepackage{refstyle}
\usepackage{amsmath}
\usepackage{amsthm}
\usepackage{amssymb}
\usepackage{stackrel}
\usepackage{setspace}
\makeatletter

\AtBeginDocument{\providecommand\lemref[1]{\ref{lem:#1}}}
\AtBeginDocument{\providecommand\propref[1]{\ref{prop:#1}}}
\AtBeginDocument{\providecommand\thmref[1]{\ref{thm:#1}}}
\AtBeginDocument{\providecommand\remref[1]{\ref{rem:#1}}}
\AtBeginDocument{\providecommand\subref[1]{\ref{sub:#1}}}
\AtBeginDocument{\providecommand\defref[1]{\ref{def:#1}}}
\AtBeginDocument{\providecommand\corref[1]{\ref{cor:#1}}}
\RS@ifundefined{subref}
  {\def\RSsubtxt{section~}\newref{sub}{name = \RSsubtxt}}
  {}
\RS@ifundefined{thmref}
  {\def\RSthmtxt{theorem~}\newref{thm}{name = \RSthmtxt}}
  {}
\RS@ifundefined{lemref}
  {\def\RSlemtxt{lemma~}\newref{lem}{name = \RSlemtxt}}
  {}

\theoremstyle{plain}
\newtheorem{thm}{\protect\theoremname}[section]
  \theoremstyle{definition}
  \newtheorem{defn}[thm]{\protect\definitionname}
  \theoremstyle{plain}
  \newtheorem{lem}[thm]{\protect\lemmaname}
  \theoremstyle{definition}
  \newtheorem{example}[thm]{\protect\examplename}
  \theoremstyle{remark}
  \newtheorem{rem}[thm]{\protect\remarkname}
  \theoremstyle{plain}
  \newtheorem{prop}[thm]{\protect\propositionname}
  \theoremstyle{plain}
  \newtheorem{cor}[thm]{\protect\corollaryname}

\@ifundefined{date}{}{\date{}}
\AtBeginDocument{\providecommand\remref[1]{\ref{rem:#1}}}

\usepackage{tikz} 
\usepackage{tocbibind}
\usepackage{ytableau}
\usepackage{authblk}

\AtBeginDocument{}
\RS@ifundefined{factref}{\newref{fact}{name = fact~,names = facts~}}{}

\AtBeginDocument{}
\RS@ifundefined{caseref}{\newref{case}{name = case~,names = cases~}}{}

\AtBeginDocument{\providecommand\remref[1]{\ref{rem:#1}}}
\RS@ifundefined{remref}{\newref{rem}{name = Remark~,names = Renarks~}}{}

\AtBeginDocument{\providecommand\defref[1]{\ref{def:#1}}}
\RS@ifundefined{defref}{\newref{def}{name =Definition~,names = Definitions~}}{}

\AtBeginDocument{\providecommand\propref[1]{\ref{prop:#1}}}
\RS@ifundefined{propref}{\newref{prop}{name =Proposition~,names = Propositions~}}{}

\AtBeginDocument{\providecommand\corref[1]{\ref{cor:#1}}}
\RS@ifundefined{corref}{\newref{cor}{name =Corollary~,names = Corollaries~}}{}

\AtBeginDocument{}
\RS@ifundefined{condref}{\newref{cond}{name = condition~,names = conditions~}}{}

\def\RSlemtxt{Lemma~}
\def\RSthmtxt{Theorem~}

\DeclareMathOperator{\Hom}{Hom}
\DeclareMathOperator{\End}{End}

\DeclareMathOperator{\tr}{\mathsf{tr}}

\DeclareMathOperator{\im}{\mathsf{im}}

\DeclareMathOperator{\Irr}{\mathsf{Irr}}

\DeclareMathOperator{\Ind}{Ind}
\DeclareMathOperator{\Res}{Res}

\DeclareMathOperator{\Set}{\text{\bf{Set}}}
\DeclareMathOperator{\type}{type}
\DeclareMathOperator{\Ex}{Ex}
\DeclareMathOperator{\Ext}{Ext}
\DeclareMathOperator{\IRR}{IRR}

\DeclareMathOperator{\FI}{{\bf FI}}
\DeclareMathOperator{\SF}{\mathcal{SF}}
\DeclareMathOperator{\inc}{\mathsf{inc}}
\DeclareMathOperator{\stab}{stab}
\DeclareMathOperator{\trace}{trace}

\usepackage{babel}
\providecommand{\corollaryname}{Corollary}
  \providecommand{\definitionname}{Definition}
  \providecommand{\examplename}{Example}
  \providecommand{\lemmaname}{Lemma}
  \providecommand{\propositionname}{Proposition}
  \providecommand{\remarkname}{Remark}
\providecommand{\theoremname}{Theorem}

\makeatother

\usepackage{babel}
  \providecommand{\corollaryname}{Corollary}
  \providecommand{\definitionname}{Definition}
  \providecommand{\examplename}{Example}
  \providecommand{\lemmaname}{Lemma}
  \providecommand{\propositionname}{Proposition}
  \providecommand{\remarkname}{Remark}
\providecommand{\theoremname}{Theorem}

\numberwithin{equation}{section}

\begin{document}

\title{The Littlewood-Richardson rule for wreath products with symmetric groups and the quiver of the category $F \wr \FI_n$}

\author{Itamar Stein\thanks{This paper is part of the author's PHD thesis, being carried out under
the supervision of Prof. Stuart Margolis. The author's research was
supported by Grant No. 2012080 from the United States-Israel Binational
Science Foundation (BSF).}}

\affil{ Department of Mathematics\\
 Bar Ilan University\\
 52900 Ramat Gan\\
 Israel}

\maketitle

\begin{abstract}
We give a new proof for the Littlewood-Richardson rule for the wreath
product $F\wr S_{n}$ where $F$ is a finite group. Our proof does
not use symmetric functions but more elementary representation theoretic
tools. We also derive a branching rule for inducing the natural embedding
of $F\wr S_{n}$ to $F\wr S_{n+1}$. We then apply the generalized
Littlewood-Richardson rule for computing the ordinary quiver of the
category $F\wr\FI_{n}$ where $\FI_{n}$ is the category of all injective
functions between subsets of an $n$-element set.
\end{abstract}

\section{Introduction}

Let $G$ be a finite group and let $H\leq G$ be a subgroup. By Maschke's
theorem, group algebras (over $\mathbb{C}$) are semisimple so every
group representation is a finite direct sum of irreducible ones. A
basic question in the representation theory of finite groups is the
following one: Let $V$ be some $H$-representation, what is the decomposition
of the inducted representation $\Ind_{H}^{G}V$ into irreducible $G$-representations?
Alternatively, let $U$ be some $G$-representation, what is the decomposition
of the restriction $\Res_{H}^{G}U$ into irreducible $H$-representations?
By Frobenius reciprocity, answering one of these questions essentially
answers the other one. Moreover, since both induction and restriction
are additive, it is enough to consider the case where $U$ or $V$
are irreducible representations. Even considering this reduction,
the question, in general, is very difficult. If $G=S_{n}$ the answer
is known for certain natural choices of $H$ and these solutions are
often called branching rules. The most classical case is where $H=S_{n-1}$
viewed as the subgroup of all permutations that fix $n$. An important
generalization is the Littlewood-Richardson rule which gives the answer
for the case $H=S_{k}\times S_{n-k}$.

Let $F$ and $G$ be finite groups such that $G$ acts on the left
of a finite set $X$. We denote by $F\wr_{X}G$ the wreath product
of $F$ and $G$. The representation theory of $F\wr_{X}G$ is a well-studied
subject (see \cite{Tullio2014} and \cite[Chapter 4]{James1981})
and the case $G=S_{n}$ with the natural action on $\{1,\ldots,n\}$
is of special importance. Finding generalizations for the branching
rules is a natural question. The ``classical'' branching rule for
inducing from $F\wr S_{n}$ to $F\wr S_{n+1}$ was found by Pushkarev
\cite{Pushkarev1997}. In this paper we generalize the Littlewood-Richardson
rule to the group $F\wr S_{n}$. After the present paper was already
circulating, we became aware of the paper \cite{Ingram2009} by Ingram,
Jing and Stitzinger, where the same result was obtained using symmetric
functions. However, our approach is different. We use only elementary
representation theoretic tools and base our proof on the explicit
description of the irreducible representations of $F\wr S_{n}$. In
\Secref{ClassicalBranchingRule} we use the generalized Littlewood-Richardson
rule to retrieve Pushkarev's result.

Then we turn to give an application to the representation theory of
a natural family of categories. Denote by $\FI$ the category of finite
sets and injective functions. The representation theory of $\FI$
is currently under active research which was initiated in \cite{Church2015}.
There is also research on the representation theory of the wreath
product $F\wr\FI$ (see \cite{Li2015,Ramos2015,Sam2014}). We will
be interested here in the finite version of this category. We denote
by $\FI_{n}$ the category of all subsets of $\{1,\ldots,n\}$ and
injective functions. In \Secref{ApplicationToCategories} we will
give a description of the ordinary quiver of the algebra of $F\wr\FI_{n}$.
The case where $F$ is the trivial group was originally done by \cite{Brimacombe2011}
and a simple proof was later given in \cite{Margolis2012}. For the
general case, we imitate the method of \cite{Margolis2012} but where
they use usual branching rule for $S_{n}$ we will use the generalization
for $F\wr S_{n}$.

\section{Preliminaries}

\subsection{Wreath product}

Throughout this paper $F$ and $G$ will be finite groups such that
$G$ acts on the left of some finite set $X$. An element $f\in F^{X}$
is a function from $X$ to $F$. Note that $F^{X}$ is a group, multiplication
being defined componentwise. We can also define a left action of $G$
on $F^{X}$ by 
\[
(g\ast f)(x)=f(g^{-1}x).
\]

It is easy to verify that this is indeed a left action. 
\begin{defn}
\label{def:WreathProductOfGroups}The wreath product of $F$ with
$G$ denoted $F\wr_{X}G$ is the semidirect product $F^{X}\rtimes G$.
In other words, it is the set $F^{X}\times G$ with multiplication
given by 
\[
(f,g)\cdot(f^{\prime},g^{\prime})=(f(g\ast f^{\prime}),gg^{\prime}).
\]

\end{defn}
If $H\leq G$ is a subgroup, we can restrict the action on $X$ to
$H$ and get the group $F\wr_{X}H$ which is a subgroup of $F\wr_{X}G$.

Clearly, $S_{n}$ acts on $\{1,\ldots,k\}$ (for $n\leq k$) by permuting
the first $n$ elements and fixing the other ones. We refer to this
action as the \emph{standard action} of $S_{n}$ on $\{1,\ldots,k\}$.
In this case we denote the wreath product with $F$ by $F\wr_{k}S_{n}$.
The focus of this paper will be the case where $k=n$ and we will
simply denote this group by $F\wr S_{n}$. There is a very natural
way to think of this group. Recall that we can identify $S_{n}$ with
the group of permutation matrices, that is, a permutation $\pi\in S_{n}$
can be identified with an $n\times n$ matrix $A$ where 
\[
A_{i,j}=\begin{cases}
1 & \pi(j)=i\\
0 & \text{otherwise}
\end{cases}.
\]
Similarly, we can identify $F\wr S_{n}$ with a group of matrices,
but here the non-zero entries can be any element of $F$. In other
words, the tuple $(f,\pi)$ is identified with the $n\times n$ matrix
$A$ where 
\[
A_{i,j}=\begin{cases}
f(i) & \pi(j)=i\\
0    & \text{otherwise}
\end{cases}.
\]
The multiplication of $F\wr S_{n}$ is then identified with matrix
multiplication when one assumes 
\[
0\cdot a=a\cdot0=0,\quad a+0=0+a=a
\]
for every $a\in F$.

The following fact will be of use (see \cite[Proposition 2.1.3]{Tullio2014}
for proof). 
\begin{lem}[Distributivity of the wreath product]
Let $G_{1}$ and $G_{2}$ be groups acting on disjoint sets $X_{1}$
and $X_{2}$ respectively, so $G_{1}\times G_{2}$ acts on the disjoint
union $X=X_{1}\dot{\cup}X_{2}$. Then 
\[
(F\wr_{X_{1}}G_{1})\times(F\wr_{X_{2}}G_{2})\cong F\wr_{X}(G_{1}\times G_{2}).
\]

\end{lem}

\subsection{Complex group representations}

We only consider representations over $\mathbb{C}$ in this paper.
A $G$-representation is a pair $(U,\rho)$ where $U$ is a finite
dimensional vector space and $\rho:G\to\End(U)$ is a group homomorphism.
This is equivalent to an action of $G$ on the vector space $U$ by
linear transformations. We will sometimes omit the homomorphism and
say that $U$ is a $G$-representation. For $u\in U$ we will usually
write $g\cdot u$ or even $gu$ instead of $\rho(g)(u)$. Let $U$
and $V$ be two $G$-representations. We say that $U$ is \emph{isomorphic}
to $V$ (and write $U\cong V$) if there is a vector space isomorphism
$T:U\to V$ such that $T(g\cdot u)=g\cdot T(u)$ for every $g\in G$
and $u\in U$. The direct sum $U\oplus V$ of two $G$-representations
is again a $G$-representation according to $g\cdot(u+v)=gu+gv$.
A subvector space $V\subseteq U$ is called a \emph{subrepresentation}
if it is closed under the action of $G$, that is, $g\cdot v\in V$
for all $g\in G$ and $v\in V$. A non-zero $G$-representation $U$
is called \emph{irreducible} if its only subrepresentations are $0$
and $U$. We denote the set of irreducible representations of $G$
(up to isomorphism) by $\Irr G$. It is well known that every $G$-representation
is a finite direct sum of irreducible representations and that the
number of different irreducible $G$-representations (up to isomorphism)
is the number of conjugacy classes of $G$. We denote the trivial
representation of any group $G$ by $\tr_{G}$. Recall that if $V$
is a $G$-representations, then $V^{\ast}=\Hom(V,\mathbb{C})$ is
also a $G$-representation with operation $(g\cdot\varphi)(v)=\varphi(g^{-1}v)$.
Let $U$ and $V$ be $G$-representations. The inner tensor product
$U\otimes V$ is again a $G$-representation with action defined by
$g\cdot(u\otimes v)=gu\otimes gv$ and extending linearly. Now, assume
that $U_{1}$ and $U_{2}$ are $G_{1}$ and $G_{2}$-representations
respectively. The outer tensor product $U_{1}\otimes U_{2}$ of $U_{1}$
and $U_{2}$ is the ($G_{1}\times G_{2}$)-representation where $(g_{1},g_{2})\cdot(u_{1}\otimes u_{2})=(g_{1}u_{1})\otimes(g_{2}u_{2})$.
Although the two types of tensor product can be distinguished by the
context we prefer using different notation for outer tensor product,
denoting it by $\boxtimes$. Likewise, the simple tensors of $U\boxtimes V$
will by denoted by $u\boxtimes v$. It is well known that $\Irr(G_{1}\times G_{2})=\{U\boxtimes V\mid U\in\Irr G_{1},V\in\Irr G_{2}\}$.
Another simple observation will be important. 
\begin{lem}
\label{lem:CommutativityOfOuterAndInnerTensorProducts} Let $U_{1}$
and $V_{1}$ ($U_{2}$ and $V_{2}$) be $G_{1}$ (respectively, $G_{2}$)-representations.
Then 
\[
(U_{1}\boxtimes U_{2})\otimes(V_{1}\boxtimes V_{2})\cong(U_{1}\otimes V_{1})\boxtimes(U_{2}\otimes V_{2})
\]
as ($G_{1}\times G_{2}$)-representations. \end{lem}
\begin{proof}
Define $T:(U_{1}\boxtimes U_{2})\otimes(V_{1}\boxtimes V_{2})\to(U_{1}\otimes V_{1})\boxtimes(U_{2}\otimes V_{2})$
by 
\[
T((u_{1}\boxtimes u_{2})\otimes(v_{1}\boxtimes v_{2}))=(u_{1}\otimes v_{1})\boxtimes(u_{2}\otimes v_{2})
\]
which clearly extends to a vector space isomorphism and also 
\begin{align*}
T((g_{1},g_{2})\cdot((u_{1}\boxtimes u_{2})\otimes(v_{1}\boxtimes v_{2}))) & =T((g_{1}u_{1}\boxtimes g_{2}u_{2})\otimes(g_{1}v_{1}\boxtimes g_{2}v_{2}))\\
 & =(g_{1}u_{1}\otimes g_{1}v_{1})\boxtimes(g_{2}u_{2}\otimes g_{2}v_{2})\\
 & =(g_{1},g_{2})\cdot T((u_{1}\boxtimes u_{2})\otimes(v_{1}\boxtimes v_{2}))
\end{align*}
as required. 
\end{proof}
The \emph{character $\chi_{U}$ }of the $G$-representation $(U,\rho)$
is the function $\chi_{U}:G\to\mathbb{C}$ defined by $\chi_{U}(g)=\trace(\rho (g))$.
Recall that the multiplicity $U\in\Irr G$ as an irreducible constituent
in some $G$-representation $V$ is given by the inner product 
\[
\langle\chi_{U},\chi_{V}\rangle=\frac{1}{|G|}\sum_{g\in G}\chi_{U}(g)\overline{\chi_{V}(g)}.
\]
Recall also that $\chi_{V^{\ast}}(g)=\overline{\chi_{V}(g)}$ and
$\chi_{U\boxtimes V}((g_{1},g_{2}))=\chi_{U}(g_{1})\chi_{V}(g_{2})$.
In order to simplify notation, we will usually omit the $\chi$ and
write $U$ also for the character of $U$. Hence the above inner product
will be written as
\[
\langle U,V\rangle=\frac{1}{|G|}\sum_{g\in G}U(g)\overline{V(g)}.
\]

\subsection{Restriction and induction}

Let $(U,\rho)$ be a $G$-representation and let $H\leq G$ be a subgroup.
The \emph{restriction} of $(U,\rho)$ to $H$ denoted $(\Res_{H}^{G}U,\Res_{H}^{G}\rho)$
is an $H$-representation defined by 
\[
\Res_{H}^{G}\rho(h)(u)=\rho(h)(u)
\]
that is, restricting the homomorphism to the subgroup $H$. Note that
$\dim\Res_{H}^{G}U\allowbreak=\dim U$ and if $U$ is an irreducible
$G$-representation then $\Res_{H}^{G}U$ does not have to be an irreducible
$H$-representation. Let $(U,\rho)$ be an $H$-representation, the\emph{
induction} to $G$ denoted $(\Ind_{H}^{G}U,\Ind_{H}^{G}\rho)$ is
the tensor product 
\[
\Ind_{H}^{G}U=\mathbb{C}G\underset{\mathbb{C}H}{\otimes}U
\]
where the $G$ action is given by 
\[
g\cdot(s\otimes u)=(gs)\otimes u
\]
where $s\in\mathbb{C}G$ and $u\in U$. However, we will usually  use
the following more concrete description. Choose $S=\{s_{1},\ldots,s_{l}\}$
to be representatives of the left cosets of $H$ in $G$ (where $l=[G:H]$).
Note that any element $g\in G$ can be written in a unique way as
$g=s_{i}h$ where $s_{i}\in S$ and $h\in H$. Every element of $\Ind_{H}^{G}U$
is a formal sum of the form

\[
\alpha_{1}(s_{1},u_{1})+\ldots+\alpha_{l}(s_{l},u_{l})
\]
where $u_{i}\in U$ and $\alpha_{i}\in\mathbb{C}$. In other words,
as a vector space $\Ind_{H}^{G}U$ is ${\displaystyle \bigoplus_{i=1}^{l}U}$,
that is, $l$ copies of $U$. The action is defined on elements of
the form $(s_{i},u)$ by 
\[
g\cdot(s_{i},u)=(s_{j},h\cdot u)
\]
where $s_{j}$ and $h$ are unique such that $gs_{i}=s_{j}h$. The
required action is given by extending linearly. Note that $\dim\Ind_{H}^{G}U=[G:H]\dim U$.
It is important to mention that the representations $\Ind_{H}^{G}U$
and $\Res_{H}^{G}V$ depend not only on the groups $G$ and $H$
but also on the specific embedding of $H$ into $G$. Hence we will
have to give the specific embeddings when discussing these representations.
Both induction and restriction are transitive and additive, that is,
if $K\leq H\leq G$ then
\[
\Ind_{H}^{G}\Ind_{K}^{H}U\cong\Ind_{K}^{G}U,\quad\Ind_{H}^{G}(U\oplus V)\cong\Ind_{H}^{G}U\oplus\Ind_{H}^{G}V
\]
and 
\[
\Res_{K}^{H}\Res_{H}^{G}U\cong\Res_{K}^{G}U,\quad\Res_{H}^{G}(U\oplus V)\cong\Res_{H}^{G}U\oplus\Res_{H}^{G}V.
\]
For restriction this is a trivial statement and for induction the
proof is \cite[Propositions 1.1.10 and 1.1.11]{Tullio2014}. An important
fact that relates induction to restriction is the following one (for
proof see \cite[Corollary 1.1.20 ]{Tullio2014}). 
\begin{thm}[Frobenius reciprocity]
\label{thm:FrobeniusReciprocity} Let $H\leq G$ and let $U$ and
$V$ be $G$ and $H$-representations respectively. Then the multiplicity
of $V$ in $\Res_{H}^{G}U$ equals the multiplicity of $U$ in $\Ind_{H}^{G}V$. 
\end{thm}
Using characters, Frobenius reciprocity can be written as the following
equality
\[
\langle\Ind_{H}^{G}V,U\rangle=\langle V,\Res_{H}^{G}U\rangle.
\]

\subsection{Representations of the symmetric group}

\ytableausetup{centertableaux}\ytableausetup{smalltableaux}

We will recall some elementary facts regarding the representation
theory of the symmetric group. More details can be found in \cite{James1981,Sagan2001}.
Recall that an \emph{integer composition} of $n$ is a tuple $\lambda=[\lambda_{1},\ldots,\lambda_{k}]$
of non-negative integers such that $\lambda_{1}+\cdots+\lambda_{k}=n$
while an \emph{integer partition of $n$ }(denoted $\lambda\vdash n$)
is an integer composition such that $\lambda_{1}\geq\lambda_{2}\geq\cdots\geq\lambda_{k}>0$.
From now on, when dealing with a partition $\lambda$ we will write
its elements in superscript $\lambda=[\lambda^{1},\ldots,\lambda^{k}]$
because we want to reserve the subscript for multipartitions. Note
that $0$ has one partition, namely the empty partition, denoted by
$\varnothing$. We can associate to any partition $\lambda$ a graphical
description called a \emph{Young diagram}, which is a table with $\lambda^{i}$
boxes in its $i$-th row. For instance, the Young diagram associated
to the partition $[3,3,2,1]$ of $9$ is: 
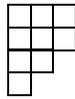
\[
\ydiagram{3,3,2,1}
\]

We will identify the two notions and regard integer partition and
Young diagram as synonyms. It is well known that irreducible representations
of $S_{n}$ are indexed by integer partitions of $n$. We denote the
irreducible representation associated to the partition $\lambda$
(also called its \emph{Specht module}) by $S^{\lambda}$. Explicit
description of $S^{\lambda}$ can be found in \cite[Section 2.3]{Sagan2001}.
It will be often convenient to draw the diagram $\lambda$ instead
of writing $S^{\lambda}$. For instance we may write

\[
\ydiagram{3}\oplus\ydiagram{2,1}
\]

instead of: $S^{\lambda}\oplus S^{\delta}$ for partitions $\lambda=[3]$
and $\delta=[2,1]$.

We will now describe several branching rules for $S_{n}$. Here the
advantage of using Young diagrams becomes clear. Recall that we can
think of $S_{n}$ as the group of all permutations of $\{1,\ldots,n+1\}$
that leave $n+1$ fixed.  Hence, we can view $S_{n}$ as a subgroup of $S_{n+1}$. We call this the standard embedding of $S_{n}$
into $S_{n+1}$.
In this case the branching rules are well known and very natural (proof
can be found in \cite[Section 2.8]{Sagan2001}). 
\begin{thm}[Classical branching rules]
Let $\lambda\vdash n$ be a Young diagram. \end{thm}
\begin{enumerate}
\item Denote by $Y^{+}(\lambda)$ the set of Young diagrams obtained from
$\lambda$ by adding one box. Then 
\[
\Ind_{S_{n}}^{S_{n+1}}S^{\lambda}=\bigoplus_{\gamma\in Y^{+}(\lambda)}S^{\gamma}.
\]

\item Similarly, denote by $Y^{-}(\lambda)$ the set of Young diagrams obtained
from $\lambda$ by removing one box. Then 
\[
\Res_{S_{n-1}}^{S_{n}}S^{\lambda}=\bigoplus_{\gamma\in Y^{-}(\lambda)}S^{\gamma}.
\]

\end{enumerate}
\begin{example}
Let $\lambda=\ydiagram{2,1}$ then

\[
\Ind_{S_{3}}^{S_{4}}S^{\lambda}=\ydiagram{3,1}\oplus\ydiagram{2,2}\oplus\ydiagram{2,1,1}
\]

and

\[
\Res_{S_{2}}^{S_{3}}S^{\lambda}=\ydiagram{2}\oplus\ydiagram{1,1}.
\]

\end{example}
We now turn our attention to the Littlewood-Richardson branching
rule. If we identify $S_{k}$ ($S_{r}$) with the group of all permutations
of $\{1,\ldots,k+r\}$ that leave $\{k+1,\ldots,k+r\}$ (respectively,
$\{1,\ldots,k\}$) fixed we can view $S_{k}\times S_{r}$ as a subgroup
of $S_{k+r}$. Given $\lambda\vdash k$ and $\delta\vdash r$ the
Littlewood\textendash Richardson rule gives the decomposition of $\Ind_{S_{k}\times S_{r}}^{S_{k+r}}(S^{\lambda}\boxtimes S^{\delta})$
into irreducible $S_{k+r}$-representations. In other words, if we
write this decomposition as 
\[
\Ind_{S_{k}\times S_{r}}^{S_{k+r}}(S^{\lambda}\boxtimes S^{\delta})=\bigoplus_{\gamma\vdash(k+r)}c_{\lambda,\delta}^{\gamma}S^{\gamma}
\]

it gives a combinatorial interpretation for the coefficients $c_{\lambda,\delta}^{\gamma}$
(called the Littlewood\textendash Richardson coefficients). The aim
of this paper is to generalize the classical branching rules and the
Littlewood-Richardson rule to the group $F\wr S_{n}$ for any finite
group $F$. Although the details of the Littlewood-Richardson rule
for $S_{n}$ will not be essential in the sequel, we give them here
for the sake of completeness. For this we have to introduce some more
notions. First we generalize the notion of a Young diagram. For $k\leq n$,
let $\lambda=(\lambda^{1},\cdots,\lambda^{r})\vdash k$ and $\gamma=(\gamma^{1},\cdots,\gamma^{s})\vdash n$
be partitions such that $\lambda^{i}\leq\gamma^{i}$ for every $1\leq i\leq r$.
The \emph{skew diagram} $\gamma/\lambda$ is the diagram obtained
by erasing the diagram $\lambda$ from the diagram $\gamma$. For
instance if $\lambda=[2,1]$ and $\gamma=[4,3,1]$ then $\gamma/\lambda$
is the skew diagram 
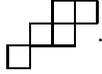
\[
\ydiagram{2+2,1+2,1}.
\]

A skew tableau is a skew diagram whose boxes are filled with numbers.
We call the original diagram the \emph{shape} of the tableau. Let
$t$ be a skew tableau with $n$ boxes such that the number of boxes
with entry $i$ is $\delta^{i}$. The \emph{content} of $t$ is the
composition $\delta=[\delta^{1},\ldots,\delta^{l}]$. We say that
a skew tableau is semi-standard if its columns are increasing and
its rows are non-decreasing. For instance

\begin{align} \label{eq:exampleTableau1} \begin{ytableau}  \none & \none & 1 & 1 \\ \none & 2 & 3 \\ 2 \end{ytableau} \end{align} 

is a semi-standard skew tableau of shape $[4,3,1]/[2,1]$ with content
$[2,2,1]$. The \emph{row word} of a skew tableau $t$ is the string
of numbers obtained by reading the entries of $t$ from right to left
and top to bottom. For instance, the row word of tableau \ref{eq:exampleTableau1}
is $11322$. A string of numbers is called a lattice permutation if
for every prefix of this string and for every number $i$, there are
no less occurrences of $i$ than occurrences of $i+1$. For instance,
the string $11322$ is not a lattice permutation since the prefix
$113$ contains one $3$ and no $2$'s. Now we can state the Littlewood-Richardson
rule (for proof see \cite[Theorem 2.8.13]{James1981}). 
\begin{thm}
\label{thm:Littlewood-Richardson rule} The Littlewood-Richardson
coefficient $c_{\lambda,\delta}^{\gamma}$ is the number of semi-standard
skew tableaux of shape $\gamma/\lambda$ with content $\delta$ whose
row word is a lattice permutation.\end{thm}
\begin{example}
If $\lambda=[2,1]$, $\delta=[3,2]$ and $\gamma=[4,3,1]$ then $c_{\lambda,\delta}^{\gamma}=2$
since there are two skew tableaux with the required properties. These
are: 
\end{example}
\begin{center} \begin{ytableau}  \none & \none & 1 & 1 \\ \none & 2 & 2 \\ 1 \end{ytableau} \qquad \begin{ytableau}  \none & \none & 1 & 1 \\ \none & 1 & 2 \\ 2 \end{ytableau} \end{center} 
\begin{rem}
Note that for many values of $\gamma$ we have $c_{\lambda,\delta}^{\gamma}=0$,
that is, $S^{\gamma}$ is not an irreducible constituent of $\Ind_{S_{k}\times S_{r}}^{S_{k+r}}(S^{\lambda}\boxtimes S^{\delta})$.
For instance, this happens for every $\gamma$ such that $\gamma^{i}<\lambda^{i}$
for some $i$. 
\end{rem}

\begin{rem}
\label{rem:LittlewoodRichardsonImpliesClassicalBranching}Note that
the classical branching rule for induction can be deduced from the
Littlewood-Richardson rule if we set $r=1$. 
\end{rem}

\section{Representation theory of $F\wr S_{n}$}

The goal of this section is to describe the irreducible representations
of $F\wr S_{n}$. We follow the approach of \cite{Tullio2014}, but
we introduce different notation that will be more convenient to our
purpose. We also prove some technical lemmas that will be of later
use.

\subsection{Inflation}
\begin{defn}
Let $G$ be a finite group and let $N\trianglelefteq G$ be a normal
subgroup. Let $(U,\rho)$ be a representation of $G/N$. We denote
by $(\overline{U},\overline{\rho})$ the $G$-representation defined
as follows. As a vector space, $\overline{U}=U$, and the $G$-action
is 
\[
\overline{\rho}(g)(u)=\rho(gN)(u)\quad\forall u\in U.
\]

Following \cite{Tullio2014} we call $\overline{U}$ the \emph{inflation}
of $U$.
\end{defn}
Note that $\dim\overline{U}=\dim U$ and if $U$ is an irreducible
$G/N$-representation then $\overline{U}$ is an irreducible $G$-representation
as well.

The specific case we will be interested in is the following. If $F$
and $G$ are finite groups, then $F^{X}\trianglelefteq F\wr_{X}G$
and $(F\wr_{X}G)/F^{X}\cong G$ so any $G$-representation $U$ can
be inflated into an ($F\wr_{X}G$)-representation with action 
\[
(f,g)\cdot u=gu.
\]

\begin{lem}
\label{lem:CommutativityOfInflationAndTensorProduct}Let $G_{1}$
and $G_{2}$ be groups acting on the disjoint sets $X_{1}$ and $X_{2}$
respectively and let $U_{1}$ and $U_{2}$ be $G_{1}$ and $G_{2}$-representations.
Note that $\overline{U_{i}}$ is an ($F\wr_{X_{i}}G_{i}$)-representation
and $\overline{U_{1}\boxtimes U_{2}}$ is an ($F\wr_{X}(G_{1}\times G_{2})$)-representation
(where $X=X_{1}\dot{\cup}X_{2}$). Then 
\[
\overline{U_{1}}\boxtimes\overline{U_{2}}\cong\overline{U_{1}\boxtimes U_{2}}
\]
as $(F\wr_{X_{1}}G_{1})\times(F\wr_{X_{2}}G_{2})\cong F\wr_{X}(G_{1}\times G_{2})$-representations.\end{lem}
\begin{proof}
As vector spaces both representations are spanned by elements of the
form $u_{1}\boxtimes u_{2}$ for $u_{i}\in U_{i}$. An element $((f_{1},g_{1}),(f_{2},g_{2}))\in F\wr_{X_{1}}G_{1}\times F\wr_{X_{2}}G_{2}$
acts in both representations by 
\[
((f_{1},g_{1}),(f_{2},g_{2}))\cdot(u_{1}\boxtimes u_{2})=g_{1}u_{1}\boxtimes g_{2}u_{2}.
\]

\end{proof}

\subsection{Conjugacy and extensions}

Let $N\trianglelefteq G$ be a normal subgroup. Then we can define
an action of $G$ on the set $\Irr N$. For every $(U,\rho)\in\Irr N$
we define $g\cdot(U,\rho)=(^{g}U,^{g}\rho)$ where 
\[
^{g}U=U,\quad^{g}\rho(n)=\rho(g^{-1}ng)
\]
for all $g\in G$ and $n\in N$.

In the case of $F^{X}\trianglelefteq F\wr G$ this action has a simple
form. Recall that the set of irreducible representations of $F^{X}$
is 
\[
\Irr F^{X}=\{{\displaystyle (\underset{x\in X}{\boxtimes}U_{x}},{\displaystyle \underset{x\in X}{\boxtimes}}\rho_{x})\mid(U_{x},\rho_{x})\in\Irr F\}.
\]

\begin{lem}[{{\cite[Lemma 2.4.1]{Tullio2014}}}]
Denote $(U,\rho)={\displaystyle (\underset{x\in X}{\boxtimes}U_{x}},{\displaystyle \underset{x\in X}{\boxtimes}}\rho_{x})\in\Irr F^{X}$
then 
\[
(f,g)\cdot U={\displaystyle \underset{x\in X}{\boxtimes}U_{g^{-1}x}}
\]
for all $(f,g)\in F\wr G$.\end{lem}
\begin{defn}
Assume that $H\leq G$ satisfies that $(U_{hx},\rho_{hx})\cong(U_{x},\rho_{x})$,
for every $x\in X$ and $h\in H$. In other words, $F\wr H$ is a subgroup
of the stabilizer of $U$. Define an ($F\wr H$)-representation $(\Ex_{H}U,\Ex_{H}\rho)$
in the following way. As a vector space, $\Ex_{H}U=U$ and the group
action is: 
\[
\Ex_{H}\rho(f,h)({\displaystyle \underset{x\in X}{\boxtimes}}u_{x})=\underset{x\in X}{\boxtimes}\rho_{h^{-1}x}(f(x))(u_{h^{-1}x})=\underset{x\in X}{\boxtimes}\rho_{x}(f(x))(u_{h^{-1}x}).
\]
$(\Ex_{H}U,\Ex_{H}\rho)$ is called the \emph{extension of $U$ with
respect to $H$}. When no ambiguity arises we write $(\widetilde{U},\widetilde{\rho})$
instead of $(\Ex_{H}U,\Ex_{H}\rho)$.\end{defn}
\begin{lem}
$(\widetilde{U},\widetilde{\rho})$ is indeed an ($F\wr H$)-representation.\end{lem}
\begin{proof}
We remark that this proof is essentially \cite[Lemma 2.4.3]{Tullio2014}.
Given $(f_{1},h_{1})$,$(f_{2},h_{2})\in F\wr H$ what we need to
prove is: 
\begin{equation}
\widetilde{\rho}((f_{1},h_{1})\cdot(f_{2},h_{2})){\displaystyle (\underset{x\in X}{\boxtimes}}u_{x})=\widetilde{\rho}((f_{1},h_{1}))(\widetilde{\rho}((f_{2},h_{2})){\displaystyle (\underset{x\in X}{\boxtimes}}u_{x})).\label{eq:HomProofOfTildeRep}
\end{equation}
The left hand side of \Eqref{HomProofOfTildeRep} is: 
\begin{align*}
\widetilde{\rho}((f_{1},h_{1})\cdot(f_{2},h_{2})){\displaystyle (\underset{x\in X}{\boxtimes}}u_{x}) & =\widetilde{\rho}((f_{1}(h_{1}\ast f_{2}),h_{1}h_{2})){\displaystyle (\underset{x\in X}{\boxtimes}}u_{x})=\\
 & ={\displaystyle \underset{x\in X}{\boxtimes}}\rho_{h_{2}^{-1}h_{1}^{-1}x}((f_{1}(h_{1}\ast f_{2}))(x))u_{h_{2}^{-1}h_{1}^{-1}x}.
\end{align*}
Since $\rho_{hx}=\rho_{x}$ for every $h\in H$ this equals 
\[
\underset{x\in X}{\boxtimes}\rho_{x}((f_{1}(h_{1}\ast f_{2}))(x))u_{h_{2}^{-1}h_{1}^{-1}x}={\displaystyle \underset{x\in X}{\boxtimes}}\rho_{x}(f_{1}(x)f_{2}(h_{1}^{-1}x))u_{h_{2}^{-1}h_{1}^{-1}x}.
\]

The right hand side of \Eqref{HomProofOfTildeRep} is

\[
\widetilde{\rho}((f_{1},h_{1}))(\tilde{\rho}((f_{2},h_{2})){\displaystyle (\underset{x\in X}{\boxtimes}}u_{x}))=\widetilde{\rho}((f_{1},h_{1}))({\displaystyle \underset{x\in X}{\boxtimes}}\rho_{h_{2}^{-1}x}(f_{2}(x))(u_{h_{2}^{-1}x})).
\]
Again, since $\rho_{h_{2}^{-1}x}=\rho_{x}$ the last expression equals
\begin{align*}
\widetilde{\rho}((f_{1},h_{1}))(\underset{x\in X}{\boxtimes}\rho_{x}(f_{2}(x))(u_{h_{2}^{-1}x})) & ={\displaystyle \underset{x\in X}{\boxtimes}}\rho_{h_{1}^{-1}x}(f_{1}(x))(\rho_{h_{1}^{-1}x}(f_{2}(h_{1}^{-1}x))(u_{h_{2}^{-1}h_{1}^{-1}x}))\\
 & ={\displaystyle \underset{x\in X}{\boxtimes}}\rho_{x}(f_{1}(x))(\rho_{x}(f_{2}(h_{1}^{-1}x))(u_{h_{2}^{-1}h_{1}^{-1}x}))\\
 & ={\displaystyle \underset{x\in X}{\boxtimes}}\rho_{x}(f_{1}(x)f_{2}(h_{1}^{-1}x))u_{h_{2}^{-1}h_{1}^{-1}x}.
\end{align*}

So we get the desired equality.\end{proof}
\begin{rem}
\label{rem:ExtensionRestrictionRemark}Let $H\leq K\leq G$ and let
$U={\displaystyle \underset{x\in X}{\boxtimes}U_{x}}$ be an $F^{X}$-representation
such that $F\wr K$ is a subgroup of the stabilizer of $U$. Note
that 
\[
\Ex_{H}U=\Res_{F\wr H}^{F\wr K}\Ex_{K}U.
\]
\end{rem}
\begin{lem}
\label{lem:CommutativityOfExtensionAndTensorProduct} Let $G_{1}$
and $G_{2}$ be groups acting on disjoint sets $X$ and $Y$ respectively.
Let $U={\displaystyle \underset{x\in X}{\boxtimes}U_{x}}$ and $V={\displaystyle \underset{y\in Y}{\boxtimes}U_{y}}$
be $F^{X}$ and $F^{Y}$-representations respectively. Assume that
$H_{1}\leq G_{1}$, $H_{2}\leq G_{2}$ are subgroups such that $F\wr H_{1}$
and $F\wr H_{2}$ are subgroups of the stabilizers of $U$ and $V$
respectively. Then 
\[
\Ex_{H_{1}}U\boxtimes\Ex_{H_{2}}V\cong\Ex_{H_{1}\times H_{2}}(U\boxtimes V)
\]
as $F\wr_{X\dot{\cup}Y}(H_{1}\times H_{2})\cong(F\wr_{X}H_{1})\times(F\wr_{Y}H_{2})$-representations.\end{lem}
\begin{proof}
Both vector spaces are spanned by elements of the form ${\displaystyle (\underset{x\in X}{\boxtimes}}u_{x})\boxtimes{\displaystyle (\underset{y\in Y}{\boxtimes}}u_{y})$
and in both cases the group action is 
\[
((f_{1},h_{1}),(f_{2},h_{2}))\cdot({\displaystyle (\underset{x\in X}{\boxtimes}}u_{x})\boxtimes{\displaystyle (\underset{y\in Y}{\boxtimes}}u_{y}))={\displaystyle (\underset{x\in X}{\boxtimes}f_{1}(x)}\cdot u_{h_{1}^{-1}x})\boxtimes{\displaystyle (\underset{y\in Y}{\boxtimes}}f_{2}(y)\cdot u_{h_{2}^{-1}y}).
\]

\end{proof}

\subsection{Irreducible representations of $F\wr S_{n}$}

We now return to the specific case of $F\wr S_{n}$. From now on,
fix some indexing for the set of irreducible representation of $F$,
say $\Irr F=\{U_{1},\ldots,U_{l}\}$. Without loss of generality,
we assume that $U_{1}$ is the trivial representation of $F$. Let
$1\leq i_{j}\leq l$ for $j=1,\ldots,n$ and let 
\[
(U,\rho)={\displaystyle (\stackrel[j=1]{n}{\boxtimes}U_{i_{j}}},{\displaystyle \stackrel[x=1]{n}{\boxtimes}}\rho_{i_{j}})
\]
be an irreducible $F^{n}$-representation. We define its \emph{type}
to be the integer composition 
\[
\type(U)=(n_{1},\ldots,n_{l})
\]
such that $n_{i}$ is the number of $j\in\{1,\ldots,n\}$ such that
$U_{i_{j}}\cong U_{i}$. Clearly ${\displaystyle \sum_{i=1}^{l}n_{i}=n}$
and $n_{i}\geq0$ for every $1\leq i\leq l$. It is clear that two
irreducible representations are in the same orbit of conjugation if
and only if they have the same type. Moreover, the stabilizer of $U$
is isomorphic to $F\wr(S_{n_{1}}\times\cdots\times S_{n_{l}})$. Given
a specific composition ${\bf n}=(n_{1},\ldots,n_{l})$ define 
\[
U_{{\bf n}}={\displaystyle U_{1}^{\boxtimes n_{1}}\boxtimes\cdots\boxtimes U_{l}^{\boxtimes n_{l}}}
\]
that is, the first $n_{1}$ representations in the product are $U_{1}$,
the next $n_{2}$ representations are $U_{2}$ etc. Clearly, the type
of $U_{{\bf n}}$ is ${\bf n}$ and the set $\{U_{{\bf n}}\mid{\bf n}$
is an integer composition of $n$\} serve as a set of representatives
for the orbits of conjugation. A tuple $\Lambda=(\lambda_{1},\ldots,\lambda_{l})$
such that $\lambda_{i}\vdash n_{i}$ for every $i$ is called a \emph{multipartition}
of $n$ with $l$ components. We will also call it a multipartition
of the composition ${\bf n}$ and denote this by $\Lambda\Vdash{\bf n}$.
We denote by $S^{\Lambda}$ the (irreducible) $S_{n_{1}}\times\cdots\times S_{n_{l}}$-representation
$S^{\Lambda}=S^{\lambda_{1}}\boxtimes\cdots\boxtimes S^{\lambda_{l}}$.
Finally we can define: 
\begin{defn}
Let ${\bf n}=(n_{1},\ldots,n_{l})$ be an integer composition of $n$
and let $\Lambda=(\lambda_{1},\ldots,\lambda_{l})$ be a multipartition
of $\mathbf{n}$. Denote by $\Phi_{\Lambda}=\Phi_{(\lambda_{1},\ldots,\lambda_{l})}$
the $F\wr S_{n}$ representation 
\[
\Ind_{F\wr(S_{n_{1}}\times\cdots\times S_{n_{l}})}^{F\wr S_{n}}(\widetilde{U_{{\bf n}}}\otimes\overline{S^{\Lambda}}).
\]
\end{defn}
\begin{thm}
\cite[Theorem 2.6.1]{Tullio2014} The set 
\[
\{\Phi_{\Lambda}\mid\mbox{{\bf n}\text{ is some integer composition of \ensuremath{n} and }}\Lambda\text{ is a multipartition of }\mbox{{\bf n}}\}
\]
is a complete list of the irreducible representations of $F\wr S_{n}$.
Moreover, $\Phi_{\Lambda}\cong\Phi_{\Lambda^{\prime}}$ if and only
if $\Lambda=\Lambda^{\prime}$. 
\end{thm}
We define a \emph{multi-Young diagram} to be a tuple of Young diagrams.
As we identify partitions and Young diagrams, we also identify multipartitions
and multi-Young diagrams. Hence, multi-Young diagrams (with $l$ components)
index the irreducible representations of $F\wr S_{n}$. For instance,
if $\Lambda=([2],[2,1],[1,1,1])$ then the irreducible representation
$\Phi_{\Lambda}$ of (say) $S_{3}\wr S_{8}$ corresponds to the multi-Young
diagram

\begin{center}
(\,\ydiagram{2}\,,\quad{}\ydiagram{2,1}\,,\quad{}\ydiagram{1,1,1}\,).
\par\end{center}
\begin{rem}
\label{rem:MultiYoungDiagramOfTrivialRepresentations}Note that $\tr_{F}$,
the trivial representation of $F=F\wr S_1$, corresponds to the multi-Young
diagram
\[
(\ydiagram{1}\,,\varnothing,\ldots,\varnothing).
\]
 
\end{rem}
Let $\lambda$ be some partition of $n$. We set $\Phi_{\lambda}^{i}=\Phi_{(\varnothing,\ldots,\varnothing,\lambda,\varnothing,\ldots,\varnothing)}$
where the non-empty partition is in the $i$-th position. Note that
\[
\Phi_{\lambda}^{i}=\widetilde{U_{i}^{\boxtimes n}}\otimes\overline{S^{\lambda}}.
\]

\begin{rem}
\label{rem:OnIrreducibleRepU_iInDifferentNotation} Since $\{U_{1},\ldots,U_{l}\}$
are irreducible representations of $F=F\wr S_{1}$ then $U_{i}$ can
be also written as $\Phi_{([1])}^{i}$. 
\end{rem}
A key observation is the following one. 
\begin{prop}
\label{prop:RecursionForIrreducibleRep}Let ${\bf n}=(n_{1},\ldots,n_{l})$
be an integer composition of $n$ and let $\Lambda=(\lambda_{1},\ldots,\lambda_{l})$
be a multipartition of ${\bf n}$. The following isomorphism holds:
\[
\Phi_{\Lambda}\cong\Ind_{F\wr S_{n_{1}}\times\cdots\times F\wr S_{n_{l}}}^{F\wr S_{n}}(\Phi_{\lambda_{1}}^{1}\boxtimes\cdots\boxtimes\Phi_{\lambda_{l}}^{l}).
\]
\end{prop}
\begin{proof}
By definition 
\[
\Phi_{\Lambda}=\Ind_{F\wr(S_{n_{1}}\times\cdots\times S_{n_{l})}}^{F\wr S_{n}}(\widetilde{U_{{\bf n}}}\otimes\overline{S^{\Lambda}})
\]
if we write this more explicitly we get 
\[
\Phi_{\Lambda}=\Ind_{F\wr S_{n_{1}}\times\cdots\times F\wr S_{n_{l}}}^{F\wr S_{n}}(\widetilde{U_{1}^{\boxtimes n_{1}}\boxtimes\cdots\boxtimes U_{l}^{\boxtimes n_{l}}})\otimes(\overline{S^{\lambda_{1}}\boxtimes\cdots\boxtimes S^{\lambda_{l}}})
\]
Using \lemref{CommutativityOfInflationAndTensorProduct} and \lemref{CommutativityOfExtensionAndTensorProduct}
this equals: 
\[
\Ind_{F\wr S_{n_{1}}\times\cdots\times F\wr S_{n_{l}}}^{F\wr S_{n}}(\widetilde{U_{1}^{\boxtimes n_{1}}}\boxtimes\cdots\boxtimes\widetilde{U_{l}^{\boxtimes n_{l}}})\otimes(\overline{S^{\lambda_{1}}}\boxtimes\cdots\boxtimes\overline{S^{\lambda_{l}}}).
\]
Now, using \lemref{CommutativityOfOuterAndInnerTensorProducts} this
equals: 
\[
\Ind_{F\wr S_{n_{1}}\times\cdots\times F\wr S_{n_{l}}}^{F\wr S_{n}}(\widetilde{U_{1}^{\boxtimes n_{1}}}\otimes\overline{S^{\lambda_{1}}})\boxtimes\cdots\boxtimes(\widetilde{U_{l}^{\boxtimes n_{l}}}\otimes\overline{S^{\lambda_{l}}})
\]
which is precisely 
\[
\Ind_{F\wr S_{n_{1}}\times\cdots\times F\wr S_{n_{l}}}^{F\wr S_{n}}(\Phi_{\lambda_{1}}^{1}\boxtimes\cdots\boxtimes\Phi_{\lambda_{l}}^{l})
\]
as required. 
\end{proof}

\section{Littlewood-Richardson rule for $F\wr S_{n}$}

In this section we generalize the standard Littlewood-Richardson rule
for $S_{n}$ to the case of $F\wr S_{n}$. As mentioned above this
is a new proof for \cite[Theorem 4.7]{Ingram2009}. Given two integers
$k$ and $r$ and two integer compositions 
\[
{\bf k}=(k_{1},\ldots,k_{l}),\quad\sum_{i=1}^{l}k_{i}=k
\]
\[
{\bf r}=(r_{1},\ldots,r_{l}),\quad\sum_{i=1}^{l}r_{i}=r
\]

let $\Lambda=(\lambda_{1},\ldots,\lambda_{l})$ and $\Delta=(\delta_{1},\ldots,\delta_{l})$
be multipartitions of ${\bf k}$ and ${\bf r}$ respectively. We want
to find the decomposition of 
\[
\Ind_{F\wr S_{k}\times F\wr S_{r}}^{F\wr S_{k+r}}(\Phi_{\Lambda}\boxtimes\Phi_{\Delta})
\]
into irreducible representations. In other words, if we write 
\[
\Ind_{F\wr S_{k}\times F\wr S_{r}}^{F\wr S_{k+r}}(\Phi_{\Lambda}\boxtimes\Phi_{\Delta})=\bigoplus_{{\bf {n}}}\bigoplus_{\Gamma\Vdash{\bf {n}}}C_{\Lambda,\Delta}^{\Gamma}\Phi_{\Gamma}
\]
where the outer sum is over all integer compositions ${\bf {n}}$
of $k+r$, we want to find the coefficients $C_{\Lambda,\Delta}^{\Gamma}$.
We start with a specific case.
\begin{prop}
\label{prop:LittlewoodRicharsonRuleSpecialCase} Let $\lambda\vdash k$
and $\delta\vdash r$ then

\[
\Ind_{F\wr S_{k}\times F\wr S_{r}}^{F\wr S_{k+r}}(\Phi_{\lambda}^{i}\boxtimes\Phi_{\delta}^{i})=\bigoplus_{\gamma\vdash(k+r)}c_{\lambda,\delta}^{\gamma}\Phi_{\gamma}^{i}
\]
where $c_{\lambda,\delta}^{\gamma}$ is the Littlewood-Richardson
coefficient. 
\end{prop}
Before proving this results we need some lemmas. 
\begin{lem}
\label{lem:InflationAndInductionCommute}Let $H\leq G$ be a subgroup
of $G$. Let $U$ be an $H$-representation then 
\[
\Ind_{F\wr_{X}H}^{F\wr_{X}G}\overline{U}\cong\overline{\Ind_{H}^{G}U}
\]
as ($F\wr_{X}G$)-representations.\end{lem}
\begin{proof}
Let $s_{1},\ldots,s_{l}$ be representatives of the $H$ cosets in
$G$. Note that $({\bf 1}_{F},s_{i})$ for $i=1,\ldots,l$ are representatives
for the $F\wr_{X}H$ cosets in $F\wr_{X}G$ where ${\bf 1}_{F}$ is
the constant function ${\bf 1}_{F}(x)=1_{F}$. Now, define 
\[
T:\overline{\Ind_{H}^{G}U}\to\Ind_{F\wr_{X}H}^{F\wr_{X}G}\overline{U}
\]
by 
\[
T((s_{i},u))=(({\bf 1}_{F},s_{i}),u)
\]
and extending linearly. Clearly, $T$ is a vector space isomorphism.
Now, given $(f,g)\in F\wr_{X}G$ 
\begin{align*}
T((f,g)\cdot(s_{i},u)) & =T(g(s_{i},u))\\
 & =T((s_{j},hu))\\
 & =(({\bf 1}_{F},s_{j}),hu)
\end{align*}
assuming that $gs_{i}=s_{j}h$ for $h\in H$. Note that 
\begin{align*}
(f,g)({\bf 1}_{F},s_{i}) & =(f(g\ast{\bf 1}_{F}),gs_{i})=(f{\bf 1}_{F},gs_{i})=\\
 & =(f,s_{j}h)=({\bf 1}_{F},s_{j})(s_{j}^{-1}\ast f,h)
\end{align*}
hence 
\begin{align*}
(f,g)\cdot T((s_{i},u)) & =(f,g)\cdot(({\bf 1}_{F},s_{i}),u)\\
 & =(({\bf 1}_{F},s_{j}),(s_{j}^{-1}\ast f,h)\cdot u)\\
 & =(({\bf 1}_{F},s_{j}),hu)
\end{align*}
so 
\[
T((f,g)\cdot(s_{i},u))=(f,g)\cdot T((s_{i},u))
\]

as required.\end{proof}
\begin{lem}[{{\cite[Proposition 1.1.15]{Tullio2014}}}]
\label{lem:InductionAndRestrictionRelation}Assume $H\leq G$ and
let $U$ ($V$) be a $G$ (respectively $H$)-representation. Then
\[
\Ind_{H}^{G}(\Res_{H}^{G}(U)\otimes V)\cong U\otimes\Ind_{H}^{G}V.
\]

\end{lem}

\begin{lem}
\label{lem:InductionrestrictionLemmaWithInflation}Let $V$ be an
$(F\wr_{X}G)$-representation. Let $H\leq G$ and let $W$ be some
$H$-representation.

Then 
\[
\Ind_{F\wr_{X}H}^{F\wr_{X}G}(\Res_{F\wr_{X}H}^{F\wr_{X}G}(V)\otimes\overline{W})\cong V\otimes\overline{\Ind_{H}^{G}W}.
\]
\end{lem}
\begin{proof}
Apply \lemref{InductionAndRestrictionRelation} and \lemref{InflationAndInductionCommute}. 
\end{proof}

\begin{proof}[Proof of \propref{LittlewoodRicharsonRuleSpecialCase}]
According to the definition 
\[
\Ind_{F\wr S_{k}\times F\wr S_{r}}^{F\wr S_{k+r}}(\Phi_{\lambda}^{i}\boxtimes\Phi_{\delta}^{i})=\Ind_{F\wr S_{k}\times F\wr S_{r}}^{F\wr S_{k+r}}((\widetilde{U_{i}^{\boxtimes k}}\otimes\overline{S^{\lambda}})\boxtimes(\widetilde{U_{i}^{\boxtimes r}}\otimes\overline{S^{\delta}})).
\]
Using \lemref{CommutativityOfOuterAndInnerTensorProducts} this equals
\[
\Ind_{F\wr S_{k}\times F\wr S_{r}}^{F\wr S_{k+r}}((\widetilde{U_{i}^{\boxtimes k}}\boxtimes\widetilde{U_{i}^{\boxtimes r}})\otimes(\overline{S^{\lambda}}\boxtimes\overline{S^{\delta}}))
\]
and by \lemref{CommutativityOfInflationAndTensorProduct} and \lemref{CommutativityOfExtensionAndTensorProduct}
this equals 
\[
\Ind_{F\wr S_{k}\times F\wr S_{r}}^{F\wr S_{k+r}}((\widetilde{U_{i}^{\boxtimes(k+r)}})\otimes(\overline{S^{\lambda}\boxtimes S^{\delta}})).
\]
If we use more precise notation, this is actually 
\[
\Ind_{F\wr S_{k}\times F\wr S_{r}}^{F\wr S_{k+r}}((\Ex_{S_{k}\times S_{r}}U_{i}^{\boxtimes(k+r)})\otimes(\overline{S^{\lambda}\boxtimes S^{\delta}}))
\]
but according to \remref{ExtensionRestrictionRemark} this equals
\[
\Ind_{F\wr S_{k}\times F\wr S_{r}}^{F\wr S_{k+r}}((\Res_{F\wr(S_{k}\times S_{r})}^{F\wr S_{k+r}}\Ex_{S_{k+r}}U_{i}^{\boxtimes(k+r)})\otimes(\overline{S^{\lambda}\boxtimes S^{\delta}})).
\]
Returning to imprecise notation, this is 
\[
\Ind_{F\wr S_{k}\times F\wr S_{r}}^{F\wr S_{k+r}}((\Res_{F\wr S_{k}\times F\wr S_{r}}^{F\wr S_{k+r}}(\widetilde{U_{i}^{\boxtimes(k+r)}}))\otimes(\overline{S^{\lambda}\boxtimes S^{\delta}}))
\]
where now 
\[
\widetilde{U_{i}^{\boxtimes(k+r)}}
\]
is an ($F\wr S_{k+r}$)-representation.

By \lemref{InductionrestrictionLemmaWithInflation} this equals 
\[
\widetilde{U_{i}^{\boxtimes(k+r)}}\otimes\overline{\Ind_{S_{k}\times S_{r}}^{S_{k+r}}(S^{\lambda}\boxtimes S^{\delta})}
\]
but according to the standard Littlewood-Richardson rule 
\[
\Ind_{S_{k}\times S_{r}}^{S_{k+r}}(S^{\lambda}\boxtimes S^{\delta})=\sum_{\gamma\vdash(k+r)}c_{\lambda,\delta}^{\gamma}S^{\gamma}.
\]
Hence our representation equals 
\[
\widetilde{U_{i}^{\boxtimes(k+r)}}\otimes\overline{\sum_{\gamma\vdash(k+r)}c_{\lambda,\delta}^{\gamma}S^{\gamma}}=\sum_{\gamma\vdash(k+r)}c_{\lambda,\delta}^{\gamma}(\widetilde{U_{i}^{\boxtimes(k+r)}}\otimes\overline{S^{\gamma}})=\sum_{\gamma\vdash(k+r)}c_{\lambda,\delta}^{\gamma}\Phi_{\gamma}^{i}
\]
as required. 
\end{proof}
Now we turn to the general case. Let ${\bf k}=(k_{1},\ldots,k_{l})$
and ${\bf r}=(r_{1},\ldots,r_{l})$ be integers compositions of $k$
and $r$ respectively. We denote ${\bf k+r}=(k_{1}+r_{1},\ldots,k_{l}+r_{l})$,
an integer composition of $k+r$. 
\begin{thm}
\label{thm:LittlewoodRichardsonRuleForWreathProduct}Let $\Lambda=(\lambda_{1},\ldots,\lambda_{l})\Vdash{\bf k}$
and $\Delta=(\delta_{1},\ldots,\delta_{l})\Vdash{\bf r}$ then 
\[
\Ind_{F\wr S_{k}\times F\wr S_{r}}^{F\wr S_{k+r}}(\Phi_{\Lambda}\boxtimes\Phi_{\Delta})=\bigoplus_{\Gamma\Vdash({\bf k}+{\bf r})}C_{\Lambda,\Delta}^{\Gamma}\Phi_{\Gamma}
\]
where 
\[
C_{\Lambda,\Delta}^{\Gamma}=\prod_{i=1}^{l}c_{\lambda_{i},\delta_{i}}^{\gamma_{i}}.
\]
\end{thm}
\begin{rem}
Note that ${\bf {k+r}}$ is the only composition of $k+r$ occurring
in this summation. 
\end{rem}
Before proving this result we need another technical lemma about induction. 
\begin{lem}
\label{lem:CommutativityOfInductionAndOuterTensorProduct}Assume $H_{1}\leq G_{1}$
($H_{2}\leq G_{2}$) and let $U_{1}$ (respectively, $U_{2}$) be
a representation of $H_{1}$ (respectively, $H_{2}$). Then 
\[
\Ind_{H_{1}\times H_{2}}^{G_{1}\times G_{2}}(U_{1}\boxtimes U_{2})\cong\Ind_{H_{1}}^{G_{1}}U_{1}\boxtimes\Ind_{H_{2}}^{G_{2}}U_{2}.
\]
\end{lem}
\begin{proof}
It is more convenient here to use the tensor product definition of
induced representation. Define a vector space isomorphism $T:\Ind_{H_{1}\times H_{2}}^{G_{1}\times G_{2}}(U_{1}\boxtimes U_{2})\to\Ind_{H_{1}}^{G_{1}}U_{1}\boxtimes\Ind_{H_{2}}^{G_{2}}U_{2}$
by 
\[
T((s_{1},s_{2})\otimes(u_{1}\boxtimes u_{2}))=(s_{1}\otimes u_{1})\boxtimes(s_{2}\otimes u_{2})
\]

and extending linearly. Now take some $(g_{1},g_{2})\in G_{1}\times G_{2}$
and note that 
\begin{align*}
T((g_{1},g_{2})\cdot((s_{1},s_{2})\otimes(u_{1}\boxtimes u_{2}))) & =T((g_{1}s_{1},g_{2}s_{2})\otimes(u_{1}\boxtimes u_{2}))\\
 & =(g_{1}s_{1}\otimes u_{1})\boxtimes(g_{2}s_{2}\otimes u_{2})\\
 & =(g_{1}\cdot(s_{1}\otimes u_{1}))\boxtimes(g_{2}\cdot(s_{2}\otimes u_{2}))\\
 & =(g_{1},g_{2})\cdot T((s_{1},s_{2})\otimes(u_{1}\boxtimes u_{2})).
\end{align*}
as required.
\end{proof}

\begin{proof}[Proof of \thmref{LittlewoodRichardsonRuleForWreathProduct}]
According to \propref{RecursionForIrreducibleRep}, the representation
\[
\Ind_{F\wr S_{k}\times F\wr S_{r}}^{F\wr S_{k+r}}(\Phi_{\Lambda}\boxtimes\Phi_{\Delta})
\]

equals

\[
\Ind_{F\wr S_{k}\times F\wr S_{r}}^{F\wr S_{k+r}}(\Ind_{F\wr S_{k_{1}}\times\cdots\times F\wr S_{k_{l}}}^{F\wr S_{k}}(\Phi_{\lambda_{1}}^{1}\boxtimes\cdots\boxtimes\Phi_{\lambda_{l}}^{l})\boxtimes\Ind_{F\wr S_{r_{1}}\times\cdots\times F\wr S_{r_{l}}}^{F\wr S_{r}}(\Phi_{\delta_{1}}^{1}\boxtimes\cdots\boxtimes\Phi_{\delta_{l}}^{l})).
\]

Using \lemref{CommutativityOfInductionAndOuterTensorProduct} this
equals 
\[
\Ind_{F\wr S_{k}\times F\wr S_{r}}^{F\wr S_{k+r}}(\Ind_{F\wr S_{k_{1}}\times\cdots\times F\wr S_{k_{l}}\times F\wr S_{r_{1}}\times\cdots\times F\wr S_{r_{l}}}^{F\wr S_{k}\times F\wr S_{r}}(\Phi_{\lambda_{1}}^{1}\boxtimes\cdots\boxtimes\Phi_{\lambda_{l}}^{l}\boxtimes\Phi_{\delta_{1}}^{1}\boxtimes\cdots\boxtimes\Phi_{\delta_{l}}^{l}))
\]

which, by transitivity of induction, equals

\[
\Ind_{F\wr S_{k_{1}}\times\cdots\times F\wr S_{k_{l}}\times F\wr S_{r_{1}}\times\cdots\times F\wr S_{r_{l}}}^{F\wr S_{k+r}}(\Phi_{\lambda_{1}}^{1}\boxtimes\cdots\boxtimes\Phi_{\lambda_{l}}^{l}\boxtimes\Phi_{\delta_{1}}^{1}\boxtimes\cdots\boxtimes\Phi_{\delta_{l}}^{l}).
\]
Rearranging we get 
\[
\Ind_{F\wr S_{k_{1}}\times F\wr S_{r_{1}}\times\cdots\times F\wr S_{k_{l}}\times F\wr S_{r_{l}}}^{F\wr S_{k+r}}((\Phi_{\lambda_{1}}^{1}\boxtimes\Phi_{\delta_{1}}^{1})\boxtimes\cdots\boxtimes(\Phi_{\lambda_{l}}^{l}\boxtimes\Phi_{\delta_{l}}^{l})).
\]
Again using transitivity we can write this as 
\[
\Ind_{F\wr S_{k_{1}+r_{1}}\times\cdots\times F\wr S_{k_{l}+r_{l}}}^{F\wr S_{k+r}}\Ind_{F\wr S_{k_{1}}\times F\wr S_{r_{1}}\times\cdots\times F\wr S_{k_{l}}\times F\wr S_{r_{l}}}^{F\wr S_{k_{1}+r_{1}}\times\cdots\times F\wr S_{k_{l}+r_{l}}}((\Phi_{\lambda_{1}}^{1}\boxtimes\Phi_{\delta_{1}}^{1})\boxtimes\cdots\boxtimes(\Phi_{\lambda_{l}}^{l}\boxtimes\Phi_{\delta_{l}}^{l}))
\]
and using \lemref{CommutativityOfInductionAndOuterTensorProduct}
this equals 
\[
\Ind_{F\wr S_{k_{1}+r_{1}}\times\cdots\times F\wr S_{k_{l}+r_{l}}}^{F\wr S_{k+r}}(\Ind_{F\wr S_{k_{1}}\times F\wr S_{r_{1}}}^{F\wr S_{k_{1}+r_{1}}}(\Phi_{\lambda_{1}}^{1}\boxtimes\Phi_{\delta_{1}}^{1})\boxtimes\cdots\boxtimes\Ind_{F\wr S_{k_{l}}\times F\wr S_{r_{l}}}^{F\wr S_{kl+r_{l}}}(\Phi_{\lambda_{l}}^{l}\boxtimes\Phi_{\delta_{l}}^{l})).
\]
According to \propref{LittlewoodRicharsonRuleSpecialCase} we get
\[
\Ind_{F\wr S_{k_{1}+r_{1}}\times\cdots\times F\wr S_{k_{l}+r_{l}}}^{F\wr S_{k+r}}((\bigoplus_{\gamma_{1}\vdash(k_{1}+r_{1})}c_{\lambda_{1},\delta_{1}}^{\gamma_{1}}\Phi_{\gamma_{1}}^{1})\boxtimes\cdots\boxtimes(\bigoplus_{\gamma_{l}\vdash(k_{l}+r_{l})}c_{\lambda_{l},\delta_{l}}^{\gamma_{l}}\Phi_{\gamma_{l}}^{l}))
\]
which equals 
\[
\bigoplus_{\gamma_{1}\vdash(k_{1}+r_{1})}\cdots\bigoplus_{\gamma_{l}\vdash(k_{l}+r_{l})}(\prod_{i=1}^{l}c_{\lambda_{i},\delta_{i}}^{\gamma_{i}}\Ind_{F\wr S_{k_{1}+r_{1}}\times\cdots\times F\wr S_{k_{l}+r_{l}}}^{F\wr S_{k+r}}(\Phi_{\gamma_{1}}^{1}\boxtimes\cdots\boxtimes\Phi_{\gamma_{l}}^{1}))
\]
which, according to \propref{RecursionForIrreducibleRep}, is precisely

\[
\bigoplus_{\Gamma\Vdash{\bf k}+{\bf r}}(\prod_{i=1}^{l}c_{\lambda_{i},\delta_{i}}^{\gamma_{i}})\Phi_{\Gamma}
\]

as required. 
\end{proof}

\section{Classical branching rules for $F\wr S_{n}$\label{sec:ClassicalBranchingRule}}

In this section we retrieve Pushkarev's result of ``classical''
branching rules for $F\wr S_{n}$ \cite[Theorem 10]{Pushkarev1997}.
Let ${\bf n}=(n_{1},\ldots,n_{r})$ be an integer composition of $n$
and let $\Lambda=(\lambda_{1},\ldots,\lambda_{l})\Vdash{\bf n}$ be
a multipartition. We want to find the decomposition into irreducible
representations of $\Ind_{F\wr S_{n}}^{F\wr S_{n+1}}\Phi_{\Lambda}$
and $\Res_{F\wr S_{n-1}}^{F\wr S_{n}}\Phi_{\Lambda}$. This is relatively
easy using the results of the previous section. We start with induction. 
\begin{thm}
\label{thm:ClassicalBranchingRule} With notation as above 
\[
\Ind_{F\wr S_{n}}^{F\wr S_{n+1}}\Phi_{\Lambda}=\bigoplus_{i=1}^{l}(\dim U_{i}\bigoplus_{\gamma\in Y^{+}(\lambda_{i})}\Phi_{(\lambda_{1},\ldots,\gamma,\ldots,\lambda_{l})})
\]
where $\gamma$ is in the $i$-th position of $(\lambda_{1},\ldots,\gamma,\ldots,\lambda_{l})$.
\end{thm}
For the proof of \thmref{ClassicalBranchingRule} we need the following
lemma. 
\begin{lem}
\label{lem:InductionAndDirectProduct} Let $U$ be an $H$-representation
and let $K$ be some group, then 
\[
\Ind_{H}^{K\times H}U\cong\mathbb{C}K\boxtimes U.
\]
\end{lem}
\begin{proof}
Clearly $\{(k,1)\mid k\in K\}$ are representatives of the $H\cong\{1_{K}\}\times H$
cosets in $K\times H$. Define $T:\Ind_{H}^{K\times H}U\to\mathbb{C}K\boxtimes U$
by 
\[
T(((k,1),u))=k\boxtimes u
\]
which is clearly a vector space isomorphism and note that 
\begin{align*}
T((k^{\prime},h^{\prime})\cdot((k,1),u)) & =T(((k^{\prime}k,1),h^{\prime}u))\\
 & =k^{\prime}k\boxtimes h^{\prime}u\\
 & =(k^{\prime},h^{\prime})\cdot(k\boxtimes u)\\
 & =(k^{\prime},h^{\prime})\cdot T(((k,1),u))
\end{align*}
so $T$ is an isomorphism of ($K\times H$)-representations. 
\end{proof}

\begin{proof}[Proof of \thmref{ClassicalBranchingRule}]
Noting that $F\wr_{n+1}S_{n}=F\wr S_{n}\times F=F\wr S_{n}\times F\wr S_{1}$
and by transitivity of induction 
\[
\Ind_{F\wr S_{n}}^{F\wr S_{n+1}}\Phi_{\Lambda}=\Ind_{F\wr S_{n}\times F\wr S_{1}}^{F\wr S_{n+1}}\Ind_{F\wr S_{n}}^{F\wr S_{n}\times F}\Phi_{\Lambda}.
\]
According to \lemref{InductionAndDirectProduct} this equals

\[
\Ind_{F\wr S_{n}\times F\wr S_{1}}^{F\wr S_{n+1}}\Phi_{\Lambda}\boxtimes\mathbb{C}F.
\]

It is well-known that the decomposition of $\mathbb{C}F$ is 
\[
\mathbb{C}F=\bigoplus_{i=1}^{l}(\dim U_{i}\cdot U_{i})
\]
so we obtain 
\[
\Ind_{F\wr S_{n}\times F\wr S_{1}}^{F\wr S_{n+1}}(\Phi_{\Lambda}\boxtimes(\bigoplus_{i=1}^{l}\dim U_{i}\cdot U_{i}))=\bigoplus_{i=1}^{l}(\dim U_{i}\Ind_{F\wr S_{n}\times F\wr S_{1}}^{F\wr S_{n+1}}(\Phi_{\Lambda}\boxtimes U_{i})).
\]
But $U_{i}=\Phi_{([1])}^{i}$ (see \remref{OnIrreducibleRepU_iInDifferentNotation})
so we can write this as 
\[
\bigoplus_{i=1}^{l}(\dim U_{i}\Ind_{F\wr S_{n}\times F\wr S_{1}}^{F\wr S_{n+1}}(\Phi_{\Lambda}\boxtimes\Phi_{([1])}^{i})).
\]
Using \thmref{LittlewoodRichardsonRuleForWreathProduct} and \remref{LittlewoodRichardsonImpliesClassicalBranching}
this is precisely the required result. 
\end{proof}
Using Frobenius reciprocity we have the following corollary for restriction. 
\begin{cor}
With notation as above 
\[
\Res_{F\wr S_{n-1}}^{F\wr S_{n}}\Phi_{\Lambda}=\bigoplus_{i=1}^{l}(\dim U_{i}\bigoplus_{\gamma\in Y^{-}(\lambda_{i})}\Phi_{(\lambda_{1},\ldots,\gamma,\ldots,\lambda_{l})})
\]

where $\gamma$ is in the $i$-th position of $(\lambda_{1},\ldots,\gamma,\ldots,\lambda_{l})$.\end{cor}
\begin{example}
Let $\Lambda$ be the multipartition associated to the multi-Young
diagram

\begin{center}
(\,\ydiagram{2}\,,\quad{}\ydiagram{2,1}\,,\quad{}\ydiagram{1,1,1}\,)
\par\end{center}

so $\Phi_{\Lambda}$ is an irreducible representation of $S_{3}\wr S_{8}$.
Assuming we have indexed $\Irr S_{3}$ such that $U_{1}$ is the trivial
representation, $U_{2}$ is the standard representation and $U_{3}$
is the alternating representation then 
\[
\Ind_{S_{3}\wr S_{8}}^{S_{3}\wr S_{9}}\Phi_{\Lambda}
\]
is associated to

\begin{center}
(\,\ydiagram{2,1}\,,\quad{}\ydiagram{2,1}\,,\quad{}\ydiagram{1,1,1}\,)
$\oplus$ (\,\ydiagram{3}\,,\quad{}\ydiagram{2,1}\,,\quad{}\ydiagram{1,1,1}\,)
$\oplus$ 2(\,\ydiagram{2}\,,\quad{}\ydiagram{3,1}\,,\quad{}\ydiagram{1,1,1}\,)
$\oplus$ 
\par\end{center}

\begin{center}
2(\,\ydiagram{2}\,,\quad{}\ydiagram{2,2}\,,\quad{}\ydiagram{1,1,1}\,)
$\oplus$ 2(\,\ydiagram{2}\,,\quad{}\ydiagram{2,1,1}\,,\quad{}\ydiagram{1,1,1}\,)
$\oplus$ (\,\ydiagram{2}\,,\quad{}\ydiagram{2,1}\,,\quad{}\ydiagram{2,1,1}\,)
$\oplus$ 
\par\end{center}

\begin{center}
(\,\ydiagram{2}\,,\quad{}\ydiagram{2,1}\,,\quad{}\ydiagram{1,1,1,1}\,). 
\par\end{center}
\end{example}

\section{Application: The quiver of the category algebra $\mathbb{C}(F\wr\FI_{n})$\label{sec:ApplicationToCategories}}

Denote by $\FI_{n}$ the category of all injective functions between
subsets of $\{1,\ldots,n\}$. In this section we apply the Littlewood-Richardson
rule for computing the ordinary quiver of the category algebra of
$F\wr\FI_{n}$, the wreath product of a finite group $F$ with $\FI_{n}$.
In the next two sections we give some preliminary background on the
wreath product of a group with a category and on quivers. In \subref{QuiverOfFWrFI}
we give the description of the quiver.

\subsection{The wreath product of a group with a category}

All categories in this paper are finite. Hence we can regard a category
$\mathcal{C}$ as a set of objects, denoted $\mathcal{C}^{0}$, and
a set of morphisms, denoted $\mathcal{C}^{1}$. If $a,b\in\mathcal{C}^{0}$
then $\mathcal{C}(a,b)$ is the set of morphisms from $a$ to $b$.
Let $g\in\mathcal{C}(a,b)$ and $g^{\prime}\in\mathcal{C}(c,d)$ be
two morphisms. Recall that the composition $g^{\prime}\cdot g$ is
defined if and only if $b=c$ and we denote this fact by by $\exists g^{\prime}\cdot g$.
A category $\mathcal{D}$ is called a \emph{subcategory }of $\mathcal{C}$
if it obtained from $\mathcal{C}$ by removing objects and morphisms.
$\mathcal{D}$ is a \emph{full subcategory} if $\mathcal{D}(a,b)=\mathcal{C}(a,b)$
for every $a,b\in\mathcal{D}^{0}$. Let $F$ be a finite group, let
$\mathcal{C}$ be a finite category and let $H:\mathcal{C}\to\Set$
be a functor from $\mathcal{C}$ to the category of finite sets. Define
a new category $\mathcal{D}$ in the following way. The set of objects
is the same as the set of objects of $\mathcal{C}$, that is, $\mathcal{D}{}^{0}=\mathcal{C}^{0}$.
Given two objects $a,b\in\mathcal{D}^{0}$, the hom-set $\mathcal{D}(a,b)$
is $\{(f,g)\mid f\in F^{H(a)},\,g\in\mathcal{C}(a,b)\}$ where $F^{H(a)}$
is the set of all functions $f:H(a)\to F$. So we can write a specific
morphism as $(f,g)$. Now, given two morphisms $(f,g)\in\mathcal{D}(a,b)$
and $(f^{\prime},g^{\prime})\in\mathcal{D}(b,c)$ the composition
is 
\[
(f^{\prime},g^{\prime})\cdot(f,g)=((f^{\prime}(H(g)))\cdot f,g^{\prime}g)
\]
where $\cdot$ is componentwise multiplication of functions in $F^{H(a)}$.
\begin{defn}
\label{def:WreathProductOfGroupAndCategory}The category $\mathcal{D}$
defined above is called the \emph{wreath product }of $F$ and $\mathcal{C}$
with respect to $H$ and it is denoted by $F\wr_{H}\mathcal{C}$.
\end{defn}
Since monoids are categories with one object, \defref{WreathProductOfGroupAndCategory}
is also a definition for the wreath product of a group $G$ with a
monoid $M$. In this case the functor $F$ is just an action of $M$
on the left of some set $X$. Hence $M$ acts on the right of $F^{X}$
in the following way. Given $g\in M$ and $f\in F^{X}$ the function
$f\ast g$ is defined by 
\[
(f\ast g)(x)=f(g\cdot x).
\]
The wreath product $F\wr_{X}M$ is then just the right semidirect
product $F^{X}\rtimes M$.
\begin{rem}
One may note that if $M$ is a group then \defref{WreathProductOfGroupAndCategory}
does not coincide with \defref{WreathProductOfGroups}. However, we
will immediately prove that the two ways to define a wreath product
of groups are isomorphic. We have to use different definition for
wreath product in this section because \defref{WreathProductOfGroups}
does not generalize well to monoids and categories.
\end{rem}
In the next lemma we denote by $F\text{wr}_{X}G$ the wreath product
of \defref{WreathProductOfGroupAndCategory} which is apriory different
from $F\wr_{X}G$ of \defref{WreathProductOfGroups}.
\begin{lem}
Let $F$ and $G$ be finite groups such that $G$ acts on the left
$X$. Then
\[
F\text{wr}_{X}G\cong F\wr_{X}G.
\]
\end{lem}
\begin{proof}
In this proof we denote by $\ast_{1}$($\ast_{2}$) the left (right)
action of $G$ on $F^{X}$ as in \defref{WreathProductOfGroups} (respectively,
\defref{WreathProductOfGroupAndCategory}). Define $T:F\wr_{X}G\to F\text{wr}_{X}G$
by
\[
T(f,g)=(f\ast_{2}g,g).
\]
Clearly, $T$ has an inverse 
\[
T^{-1}(f,g)=(f\ast_{2}g^{-1},g).
\]
Moreover, note that
\begin{align*}
T((f,g)\cdot(f^{\prime},g^{\prime})) & =T(f\cdot(g\ast_{1}f^{\prime}),gg^{\prime})=T(f\cdot(f^{\prime}\ast_{2}g^{-1}),gg^{\prime})\\
 & =((f\ast_{2}(gg^{\prime}))\cdot(f^{\prime}\ast_{2}g^{\prime}),gg^{\prime})
\end{align*}
while 
\[
T((f,g))\cdot T((f^{\prime},g^{\prime}))=(f\ast_{2}g,g)\cdot(f^{\prime}\ast_{2}g^{\prime},g^{\prime})=((f\ast_{2}(gg^{\prime}))\cdot(f^{\prime}\ast_{2}g^{\prime}),gg^{\prime})
\]
so $T$ is also a group homomorphism as required.
\end{proof}

\subsection{The ordinary quiver of an EI-category algebra}

Recall that a unital \emph{$\mathbb{C}$-algebra} is a unital ring
$A$ that is also a vector space over $\mathbb{C}$ such that $c(ab)=(ca)b=a(cb)$
for all $c\in\mathbb{C}$ and $a,b\in A$. The algebras that are of
interest for us in this section are category algebras. Let $\mathcal{D}$
be a finite category. The \emph{category algebra} $\mathbb{C}\mathcal{D}$
is the $\mathbb{C}$-vector space with basis the morphisms of the
category, that is, all formal linear combinations
\[
\{c_{1}g_{1}+\ldots+c_{k}g_{k}\mid c_{i}\in\mathbb{C},\,g_{i}\in\mathcal{D}^{1}\}
\]
with multiplication being linear extension of
\[
g^{\prime}\cdot g=\begin{cases}
g^{\prime}g & \exists g^{\prime}\cdot g\\
0 & \text{otherwise}
\end{cases}.
\]

A quiver is a non-directed graph where multiple edges and loops are
permitted. The (ordinary) quiver $Q$ of an algebra $A$ is a quiver
that contains information about the algebra's representations. The
exact definition is as follows. The vertices of $Q$ are in a one-to-one
correspondence with the set $\Irr A$ of all irreducible representations
of $A$ (up to isomorphism). Given two irreducible representations
$U$ and $V$ the number of edges (more often called \emph{arrows})
from $U$ to $V$ is 
\[
\dim\Ext^{1}(U,V).
\]
For the sake of simplicity, if $Q$ is the quiver of the algebra of
$\mathcal{D}$ we will call it simply the quiver of $\mathcal{D}$. 

When considering quivers of categories we can restrict our discussion
to a special kind of categories. Two categories $\mathcal{C}$ and
$\mathcal{D}$ are called \emph{equivalent} if there are functors $\mathcal{F}:\mathcal{C}\to \mathcal{D}$ and $\mathcal{G}:\mathcal{D}\to \mathcal{C}$  such that $\mathcal{F}\mathcal{G}\cong 1_{\mathcal{D}}$ and $\mathcal{G}\mathcal{F}\cong 1_\mathcal{C}$ where $\cong$ is natural isomorphism of functors. It is well known that $\mathcal{C}$ and $\mathcal{D}$ are equivalent if and only if  there is a fully faithful
and essentially surjective functor from $\mathcal{C}$ to $\mathcal{D}$. If $\mathcal{C}$ and $\mathcal{D}$ are equivalent
categories then they have the same quiver (since their algebras are
Morita equivalent, see \cite[Proposition 2.2]{Webb2007}). A category
$\mathcal{C}$ is called \emph{skeletal} if no two objects of $\mathcal{C}$
are isomorphic. Note that any category $\mathcal{C}$ is equivalent
to some (unique) skeletal category called its \emph{skeleton}. The
skeleton of $\mathcal{C}$ is the full subcategory having one object
from every isomorphism class of $\mathcal{C}$. So we can restrict
ourselves to discussing skeletal categories. 

There is a special kind of categories whose quiver has a more concrete
description. This description was discovered independently by Li \cite{Li2011}
and by Margolis and Steinberg \cite{Margolis2012}. For explaining
it we need more definitions from category theory. A category $\mathcal{D}$
is called an \emph{EI-category} if every endomorphism is an isomorphism.
In other words, every endomorphism monoid $\mathcal{D}(a,a)$ of this
category is a group. A morphism $g\in\mathcal{D}^{1}$ of an EI-category
is called \emph{irreducible} if it is not an isomorphism but whenever
$g=g^{\prime}g^{\prime\prime}$, either $g^{\prime}$ or $g^{\prime\prime}$
is an isomorphism. The set of irreducible morphisms from $a$ to $b$
is denoted $\IRR A(a,b)$. The quiver of skeletal $EI$-categories
is described in the following theorem, which is \cite[Theorem 4.7]{Li2011}
or \cite[Theorem 6.13]{Margolis2012} for the case of the field of
complex numbers.
\begin{thm}
\label{thm:QuiverOfEICategories}Let $\mathcal{D}$ be a finite skeletal
EI-category and denote by $Q$ the quiver of $\mathcal{D}$. Let $\mathbb{C}\IRR\mathcal{D}(a,b)$
denote the $\mathbb{C}$-vector space spanned by the set $\IRR\mathcal{D}(a,b)$.
It is also an $\mathcal{D}(b,b)\times\mathcal{D}(a,a)$-representation
according to
\[
(h^{\prime},h)\cdot g=h^{\prime}gh^{-1}
\]
for $(h^{\prime},h)\in\mathcal{D}(b,b)\times\mathcal{D}(a,a)$ and
$g\in\IRR\mathcal{D}(a,b)$. Then\end{thm}
\begin{enumerate}
\item The vertex set of $Q$ is ${\displaystyle \bigsqcup_{a\in\mathcal{D}^{0}}}\Irr\mathcal{D}(a,a)$.
\item If $V\in\Irr(\mathcal{D}(a,a))$ and $U\in\Irr(\mathcal{D}(b,b))$,
then the number of arrows from $V$ to $U$ is the multiplicity of
$U\boxtimes V^{\ast}$ as an irreducible constituent of the $\mathcal{D}(b,b)\times\mathcal{D}(a,a)$-representation
$\mathbb{C}\IRR\mathcal{D}(a,b)$.
\end{enumerate}

\subsection{The quiver of the category $F\wr\FI_{n}$\label{sub:QuiverOfFWrFI}}

As mentioned above, we denote by $\FI_{n}$ the category of all injective
functions between subsets of $\{1,\ldots,n\}$. In other words, the
objects of $\FI_{n}$ are subsets of $\{1,\ldots,n\}$ and given two
objects $A$ and $B$ the hom-set $\FI_{n}(A,B)$ contains all the
injective functions from $A$ to $B$. Note that $\varnothing$ is
an initial object of this category, that is, for every $A\subseteq\{1,\ldots,n\}$
there is a unique empty function from $\varnothing$ to $A$. We are
interested in the category $F\wr_{H}\FI_{n}$ where $H:\FI_{n}\to\Set$
is the inclusion functor. We will omit the $H$ and denote this category
by $F\wr\FI_{n}$. This category has a natural description using matrices
similar to the description of $F\wr S_{n}$. We can identify the hom-set
$F\wr\FI_{n}(A,B)$ with a set of matrices whose rows are indexed
by elements of $B$ and columns are indexed by elements of $A$. The
matrix $M^{(f,g)}$ identified with $(f,g)\in F\wr\FI_{n}(A,B)$ is
defined by
\[
M_{i,j}^{(f,g)}=\begin{cases}
f(j) & g(j)=i\\
0 & \text{otherwise}
\end{cases}
\]
where $i\in B$ and $j\in A$. $g$ is a total function so $M^{(f,g)}$
has no zero columns. Moreover, since $g$ is an injective function
$M^{(f,g)}$ is column and row monomial, that is, every row and column
contains at most one non-zero element. Hence $F\wr\FI_{n}(A,B)$ can
be identified with the set of all column and row monomial matrices
over $F$ without zero columns where the columns are indexed by $A$
and the rows are indexed by $B$. Composition of morphisms then
corresponds to matrix multiplication. Note that the multiplication
$M^{(f^{\prime},g^{\prime})}\cdot M^{(f,g)}$ of two matrices of this
form where $(f,g)\in F\wr\FI_{n}(A,B)$ and $(f^{\prime},g^{\prime})\in F\wr\FI_{n}(C,D)$
is defined if and only if $B=C$. In other words, multiplication is
defined if and only if the columns of $M^{(f^{\prime},g^{\prime})}$
and the rows of $M^{(f,g)}$ are indexed by the same set. It is easy
to see that any endomorphism monoid $F\wr\FI_{n}(A,A)$ of this category
is isomorphic to the group $F\wr S_{A}$ hence this category is actually
an EI-category. Our goal is to describe the quiver of $F\wr\FI_{n}$.
The case where $F$ is the trivial group was originaly done in \cite[Theorem 8.1.2]{Brimacombe2011}.
A different computation of this case using \thmref{QuiverOfEICategories}
is done in \cite[Example 6.15]{Margolis2012} and here we merely imitate
their method. As explained in the previous section we can work with
the skeleton of $F\wr\FI_{n}$. For the sake of simplicity we will
denote this skeleton by $\SF_{n}$. It is clear that two objects $A$
and $B$ of $F\wr\FI_{n}$ are isomorphic if and only if $|A|=|B|$.
Hence we can identify the skeleton $\SF_{n}$ with the full subcategory
of $F\wr\FI_{n}$ whose objects are the empty set and $\{1,\ldots,k\}$
for $k=1,\ldots,n$. So we can identify the objects of $\SF_{n}$
with $0,\ldots,n$. Now the hom-set $\SF_{n}(k,r)$ is identified
with all the $r\times k$ matrices over $F$ which are column and
row monomial and without zero columns. Composition of morphisms
then corresponds to matrix multiplication as explained above. By
\thmref{QuiverOfEICategories} we know that the vertices of the quiver
of $\SF_{n}$ are in one-to-one correspondence with irreducible representation
of the endomorphism groups, which are $F\wr S_{k}$ for $0\leq k\leq n$.
In other words:
\begin{cor}
Let $F$ be a group with $l$ distinct irreducible representations.
The vertices in the quiver of $F\wr\FI_{n}$ can be identified with
all the multi-Young diagrams with $k$ boxes and $l$ components where
$0\leq k\leq n$.
\end{cor}
The next step is to identify the irreducible morphisms of $\SF_{n}$.
\begin{lem}
\label{lem:IrreducibleMorphismsOfFI}The irreducible morphisms of
$\SF_{n}$ are precisely the morphisms from $k$ to $k+1$ for $0\leq k\leq n-1$.
In other words,
\[
\IRR\SF_{n}(k,r)=\begin{cases}
\SF_{n}(k,r) & r=k+1\\
0 & \text{otherwise}
\end{cases}.
\]
\end{lem}
\begin{proof}
It is clear that any morphism in $\SF_{n}(k,k+1)$ is irreducible.
On the other hand, take some morphism $(f,g)\in\SF_{n}(k,r)$ and
assume that $k+1<r$. Choose $j\in\{1,\ldots,r\}$ not in the image
of $g$. Define $\inc:\{1,\ldots,k\}\to\{1,\ldots,k+1\}$ to be the
inclusion function and $g^{\prime}:\{1,\ldots,k+1\}\to\{1,\ldots,r\}$
is the function defined by 
\[
g^{\prime}(i)=\begin{cases}
g(i) & i\leq k\\
j & i=k+1
\end{cases}.
\]
It is clear that $g^{\prime}$ and $\inc$ are not bijections and
that $g=g^{\prime}\circ\inc$. Denote by ${\bf 1}_{F}$ the constant
function ${\bf 1}_{F}:\{1,\ldots,k+1\}\to F$ defined by ${\bf 1}_{F}(i)=1_{F}$
for $i=1,\ldots,k+1$. Since $g^{\prime}$ and $\inc$ are not bijections
it is clear that $({\bf 1}_{F},g^{\prime})$ and $(f,\inc)$ are not
isomorphisms. Moreover,$({\bf 1}_{F},g^{\prime})\cdot(f,\inc)=(f,g)$
so $(f,g)$ is not an irreducible morphism as required.
\end{proof}
From \thmref{QuiverOfEICategories} and \lemref{IrreducibleMorphismsOfFI}
we can immediately deduce the following corollary.
\begin{cor}
\label{cor:AllArrowsInTheQuiverAreOneStepUp}Let $V\in\Irr F\wr S_{k}$
and $U\in\Irr F\wr S_{r}$ be two vertices in the quiver of $F\wr\FI_{n}$
such that $r\neq k+1$ then there are not arrows from $V$ to $U$.
\end{cor}
It is left to consider the situation where $r=k+1$. We have to study
the representation $\mathbb{C}\IRR(\SF_{n}(k,k+1))=\mathbb{C}\SF_{n}(k,k+1)$
under the action described in \thmref{QuiverOfEICategories}. This
is a permutation representation, i.e., it is a linearization of the
action of $F\wr S_{k}\times F\wr S_{k+1}$ on the set $\SF_{n}(k,k+1)$
given by
\[
((h^{\prime},\pi^{\prime}),(h,\pi))\cdot(f,g)=(h^{\prime},\pi^{\prime})\cdot(f,g)\cdot(h,\pi)^{-1}.
\]

\begin{lem}
The above action is transitive.\end{lem}
\begin{proof}
Chose some $(f,g)\in\SF_{n}(k,k+1)$ and let $j\in\{1,\ldots,k+1\}$
be the only element not in the image of $g$. Define $\pi^{\prime}\in S_{k+1}$
by 
\[
\pi^{\prime}(i)=\begin{cases}
g^{-1}(i) & i\neq j\\
k+1 & i=j
\end{cases}.
\]
Recall that $g$ is injective so $g^{-1}(i)$ is well defined if $i\in\im g$.
It is clear that $\pi^{\prime}g=\inc:\{1,\ldots,k\}\to\{1,\ldots,k+1\}$.
Now define $h^{\prime}:\{1,\ldots,k+1\}\to F$ by 
\[
h^{\prime}(i)=\begin{cases}
(f(g^{-1}(i)))^{-1} & i\neq j\\
1 & i=j
\end{cases}.
\]
It is easy to see that $(h^{\prime},\pi^{\prime})\cdot(f,g)=({\bf 1}_{F},\inc)$
hence the action is transitive (even if we multiply only on the left).
\end{proof}
It is well-known that if the action of $G$ on some set $X$ is transitive
then the permutation representation $\mathbb{C}X$ is $\Ind_{K}^{G}(\tr_{K})$
where $K=\stab(x)$ is the stabilizer of some $x\in X$ and $\tr_{K}$
is its trivial representation.

So our representation is also of this form. We want to understand
better the stabilizer of some $x\in\SF_{n}(k,k+1)$. It is convenient
to choose $x=({\bf 1}_{F},\inc)$ and to use the matrix interpretation
discussed above.
\begin{lem}
Choose $x=({\bf 1}_{F},\inc)\in\SF_{n}(k,k+1)$. The stabilizer $\stab(x)$
is isomorphic to $(F\wr S_{k})\times F$.\end{lem}
\begin{proof}
$({\bf 1}_{F},\inc)$ is identified with the $(n+1)\times n$ matrix
with $1$ along its main diagonal and $0$ elsewhere.
\[
({\bf 1}_{F},\inc)=\left(\begin{array}{ccc}
1 & 0 & 0\\
0 & \ddots & 0\\
0 & 0 & 1\\
0 & \cdots & 0
\end{array}\right)=\left(\begin{array}{ccc}
\\
 & I\\
\\
\hline 0 & \cdots & 0
\end{array}\right).
\]

It is easy to see that given any matrix $A\in F\wr S_{k}$ if we want
some $B\in F\wr S_{k+1}$ such that 
\[
B\left(\begin{array}{ccc}
\\
 & I\\
\\
\hline 0 & \cdots & 0
\end{array}\right)A=\left(\begin{array}{ccc}
\\
 & I\\
\\
\hline 0 & \cdots & 0
\end{array}\right)
\]
then $B$ must be of the form 
\[
B=\left(\begin{array}{ccc|c}
 &  &  & 0\\
 & A^{-1} &  & \vdots\\
 &  &  & 0\\
\hline 0 & \cdots & 0 & a
\end{array}\right)=A^{-1}\oplus(a)
\]
where $a$ can be any element of $F$. Hence 
\[
\stab(({\bf 1}_{F},\inc))=\{(A\oplus(a),A)\mid A\in F\wr S_{k},\quad a\in F\}\cong(F\wr S_{k})\times F
\]
as required.\end{proof}
\begin{prop}
\label{prop:NumberOfArrowsByBranchingRule}Let $V\in\Irr F\wr S_{k}$
and $U\in\Irr F\wr S_{k+1}$ identified with two vertices of the quiver.
The number of arrows from $V$ to $U$ is the multiplicity of $U$
as an irreducible constituent in the $F\wr S_{k+1}$-representation
$\Ind_{(F\wr S_{k})\times F}^{F\wr S_{k+1}}(V\boxtimes\tr_{F})$. \end{prop}
\begin{proof}
Denote $K=\stab(({\bf 1}_{F},\inc))$. According to \thmref{QuiverOfEICategories}
and the above discussion, the required number is the multiplicity
of $U\boxtimes V^{\ast}$ as an irreducible constituent in the $F\wr S_{k+1}\times F\wr S_{k}$-representation
$\Ind_{K}^{F\wr S_{k+1}\times F\wr S_{k}}\tr_{K}$ where $\tr_{K}$
is the trivial representation of $K$. Using inner product of characters,
and recalling that we use the same notation for the representation
and its character, this number is
\[
\langle U\boxtimes V^{\ast},\Ind_{K}^{F\wr S_{k+1}\times F\wr S_{k}}\tr_{K}\rangle.
\]
By Frobenius reciprocity this equals 
\begin{align*}
\langle U\boxtimes V^{\ast},\Ind_{K}^{F\wr S_{k+1}\times F\wr S_{k}}\tr_{K}\rangle & =\langle\Res_{K}^{F\wr S_{k+1}\times F\wr S_{k}}(U\boxtimes V^{\ast}),\tr_{K}\rangle.
\end{align*}

Recall that $K=\{(A\oplus(a),A)\mid A\in F\wr S_{k},\quad a\in F\}$
so 

\begin{align*}
\langle\Res_{K}^{F\wr S_{k+1}\times F\wr S_{k}}(U\boxtimes V^{\ast}),\tr_{K}\rangle. & =\frac{1}{|K|}\sum_{(A\oplus(a),A)\in K}U\boxtimes V^{\ast}(A\oplus(a),A)\\
 & =\frac{1}{|K|}\sum_{(A\oplus(a),A)\in K}U(A\oplus(a))V^{\ast}(A)\\
 & =\frac{1}{|K|}\sum_{(A\oplus(a),A)\in K}U(A\oplus(a))\overline{V(A)}
\end{align*}
We want to think of the last term as an inner product of $(F\wr S_{k})\times F$-representations,
but neither $U$ nor $V$ is an $(F\wr S_{k})\times F$-representations.
However, 
\[
U(A\oplus(a))=\Res_{F\wr S_{k}\times F}^{F\wr S_{k+1}}U(A\oplus(a))
\]
and
\[
V(A)=V(A)\tr_{F}(a)=(V\boxtimes\tr_{F})(A,a)
\]
as a $K\cong(F\wr S_{k})\times F$ representation. Hence
\begin{align*}
\frac{1}{|K|}\sum_{(A\oplus(a),A)\in K}U(A\oplus(a))\overline{V(A)}
\end{align*}

equals

\begin{align*}
\frac{1}{|(F\wr S_{k})\times F|}\sum_{(A,a)\in F\wr S_{k}\times F}(\Res_{(F\wr S_{k})\times F}^{F\wr S_{k+1}}U)(A\oplus(a))\overline{(V\boxtimes\tr_{F})(A,a)}
\end{align*}

which is the inner product
\[
\langle\Res_{(F\wr S_{k})\times F}^{F\wr S_{k+1}}U,V\boxtimes\tr_{F}\rangle.
\]

Using again Frobenius reciprocity this equals 
\[
\langle U,\Ind_{K}^{F\wr S_{k+1}}(V\boxtimes\tr_{F})\rangle
\]
which is precisely the required number.
\end{proof}
Clearly $K\cong(F\wr S_{k})\times F=(F\wr S_{k})\times(F\wr S_{1})$.
Note also that the embedding $K\hookrightarrow F\wr S_{k+1}$ is precisely
the standard embedding of the Littlewood-Richardson rule. By \thmref{LittlewoodRichardsonRuleForWreathProduct}
we can conclude:
\begin{thm}
The vertices of the quiver of $F\wr\FI_{n}$ are in one-to-one correspondence
with multi-Young diagrams with with $k$ boxes and $l$ components
where $l=|\Irr F|$ and $0\leq k\leq n$. Let ${\bf k}=(k_{1},\ldots k_{l})$
and ${\bf r}=(r_{1},\ldots,r_{l})$ be two integer compositions of
$k$ and $r$ respectively and let $\Lambda=(\lambda_{1},\ldots,\lambda_{l})\Vdash\mbox{{\bf k}}$
and $\Delta=(\delta{}_{1},\ldots,\delta_{l})\Vdash\mbox{{\bf r}}$
be multipartitions of ${\bf k}$ and $\mathbf{r}$ respectively. There
can be no more than one arrow from $\Phi_{\Lambda}$ to $\Phi_{\Delta}$.
There is an arrow if and only if the following holds:
\begin{itemize}
\item $r=k+1$.
\item $r_{1}=k_{1}+1$ and $r_{i}=k_{i}$ for $2\leq i\leq l$.
\item $\lambda_{1}$ is obtained from $\delta_{1}$ by adding one box, and
$\lambda_{i}=\delta_{i}$ for $2\leq i\leq l$.
\end{itemize}
\end{thm}
\begin{proof}
We have already seen that there are no arrows from $\Phi_{\Lambda}$
to $\Phi_{\Delta}$ unless $r=k+1$ (\corref{AllArrowsInTheQuiverAreOneStepUp}).
If $r=k+1$ then \propref{NumberOfArrowsByBranchingRule} implies
that the number of arrows from $\Phi_{\Lambda}$ to $\Phi_{\Delta}$
is the multiplicity of $\Phi_{\Delta}$ as an irreducible constituent
in $\Ind_{F\wr S_{k}\times F}^{F\wr S_{k+1}}(\Phi_{\Lambda}\boxtimes\tr_{F})=\Ind_{F\wr S_{k}\times F\wr S_{1}}^{F\wr S_{k+1}}(\Phi_{\Lambda}\boxtimes\tr_{F})$.
Recall that by \remref{MultiYoungDiagramOfTrivialRepresentations}
the multi-Young diagram corresponds to $\tr_{F}$ is
\[
(\ydiagram{1}\,,\varnothing,\ldots,\varnothing)
\]
so the result follows immediately from the Littlewood-Richardson rule
for $F\wr\nobreak S_{n}$.\end{proof}
\begin{example}
The quiver of the category $S_{3}\wr \FI_{2}$ is given in the following
figure:

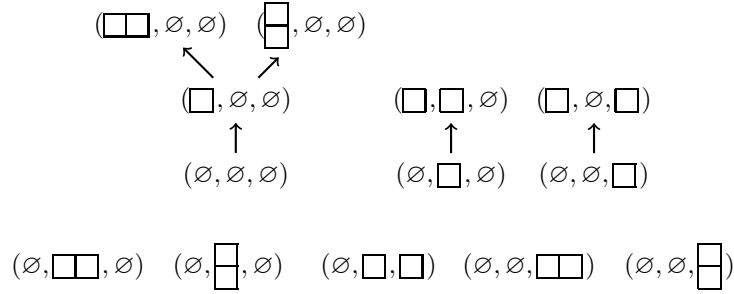
\begin{center}\begin{tikzpicture}\path (0,-2) node (S0_0_0) {$(\varnothing,\varnothing,\varnothing)$}; \path (0,-1) node (S1_0_0) {$(\ydiagram{1}\,,\varnothing,\varnothing)$}; \path (-1,0) node (S21_0_0) {$(\ydiagram{2}\,,\varnothing,\varnothing)$}; \path (1,0) node (S22_0_0) {$(\ydiagram{1,1}\,,\varnothing,\varnothing)$};\draw[thick,->] (S0_0_0)--(S1_0_0); \draw[thick,->] (S1_0_0)--(S21_0_0); \draw[thick,->] (S1_0_0)--(S22_0_0); \end{tikzpicture} \begin{tikzpicture}\path (0,-1) node (S0_1_0) {$(\varnothing,\ydiagram{1}\,,\varnothing)$}; \path (0,0) node (S1_1_0) {$(\ydiagram{1}\,,\ydiagram{1}\,,\varnothing)$}; \draw[thick,->] (S0_1_0)--(S1_1_0); \end{tikzpicture} \begin{tikzpicture}\path (0,-1) node (S0_0_1) {$(\varnothing,\varnothing,\ydiagram{1}\,)$}; \path (0,0) node (S1_0_1) {$(\ydiagram{1}\,,\varnothing,\ydiagram{1}\,)$}; \draw[thick,->] (S0_0_1)--(S1_0_1); \end{tikzpicture} \end{center}

\begin{center}\begin{tikzpicture}\path (-4,0) node (S0_21_0) {$(\varnothing,\ydiagram{2}\,,\varnothing)$}; \path (-2,0) node (S0_22_0) {$(\varnothing,\ydiagram{1,1}\,,\varnothing)$}; \path (0,0) node (S0_1_1) {$(\varnothing,\ydiagram{1}\,,\ydiagram{1}\,)$}; \path (2,0) node (S0_0_21) {$(\varnothing,\varnothing,\ydiagram{2}\,)$}; \path (4,0) node (S0_0_22) {$(\varnothing,\varnothing,\ydiagram{1,1}\,)$}; \end{tikzpicture} \end{center}
\end{example}
Clearly, two multipartitions $\Lambda=(\lambda_{1},\ldots,\lambda_{l})$
and $\Delta=(\delta{}_{1},\ldots,\delta_{l})$ are in the same connected
component if and only if $\lambda_{i}=\delta_{i}$ for $i=2,\ldots,l$.
Hence connected components can be parametrized by multipartitions
of $k$ with $l-1$ components where $k=0,\ldots,n$. Denote by $P_{l}(n)$
the number of multipartitions of $n$ with $l$ components. A generating
function for this sequence and other formulas can be found in \cite{Andrews2008}.
The following result is immediate.
\begin{cor}
Let $F$ be a non-trivial finite group and denote $l=|\Irr F|$. Then
the quiver of $F\wr\FI_{n}$ has 
\[
\sum_{k=0}^{n}P_{l-1}(k)
\]
connected components.
\end{cor}
\textbf{Acknowledgements: }The author is grateful to Tullio Ceccherini-Silberstein
for examining this work and for his very helpful remarks. 

\bibliographystyle{plain}
\bibliography{library}

\begin{thebibliography}{10}

\bibitem{Andrews2008}
George~E. Andrews.
\newblock A survey of multipartitions: congruences and identities.
\newblock In {\em Surveys in number theory}, volume~17 of {\em Dev. Math.},
  pages 1--19. Springer, New York, 2008.

\bibitem{Brimacombe2011}
Bridget~Helen Brimacombe.
\newblock {\em The {R}epresentation {T}heory of the {I}ncidence {A}lgebra of an
  {I}nverse {S}emigroup}.
\newblock ProQuest LLC, Ann Arbor, MI, 2011.
\newblock Thesis (Ph.D.)--Carleton University (Canada).

\bibitem{Tullio2014}
Tullio Ceccherini-Silberstein, Fabio Scarabotti, and Filippo Tolli.
\newblock {\em Representation theory and harmonic analysis of wreath products
  of finite groups}, volume 410 of {\em London Mathematical Society Lecture
  Note Series}.
\newblock Cambridge University Press, Cambridge, 2014.

\bibitem{Church2015}
Thomas Church, Jordan~S. Ellenberg, and Benson Farb.
\newblock F{I}-modules and stability for representations of symmetric groups.
\newblock {\em Duke Math. J.}, 164(9):1833--1910, 2015.

\bibitem{Ingram2009}
Frank Ingram, Naihuan Jing, and Ernie Stitzinger.
\newblock Wreath product symmetric functions.
\newblock {\em Int. J. Algebra}, 3(1-4):1--19, 2009.

\bibitem{James1981}
Gordon James and Adalbert Kerber.
\newblock {\em The representation theory of the symmetric group}, volume~16 of
  {\em Encyclopedia of Mathematics and its Applications}.
\newblock Addison-Wesley Publishing Co., Reading, Mass., 1981.
\newblock With a foreword by P. M. Cohn, With an introduction by Gilbert de B.
  Robinson.

\bibitem{Li2011}
Liping Li.
\newblock A characterization of finite {EI} categories with hereditary category
  algebras.
\newblock {\em J. Algebra}, 345:213--241, 2011.

\bibitem{Li2015}
Liping Li.
\newblock Upper bounds of homological invariants of {FI\_G}-modules.
\newblock {\em arXiv preprint arXiv:1512.05879}, 2015.

\bibitem{Margolis2012}
Stuart Margolis and Benjamin Steinberg.
\newblock Quivers of monoids with basic algebras.
\newblock {\em Compos. Math.}, 148(5):1516--1560, 2012.

\bibitem{Pushkarev1997}
I.~A. Pushkarev.
\newblock On the theory of representations of the wreath products of finite
  groups and symmetric groups.
\newblock {\em Zap. Nauchn. Sem. S.-Peterburg. Otdel. Mat. Inst. Steklov.
  (POMI)}, 240(Teor. Predst. Din. Sist. Komb. i Algoritm. Metody. 2):229--244,
  294--295, 1997.

\bibitem{Ramos2015}
Eric Ramos.
\newblock Homological invariants of {FI}-modules and {FI}\_{G}-modules.
\newblock {\em arXiv preprint arXiv:1511.03964}, 2015.

\bibitem{Sagan2001}
Bruce~E. Sagan.
\newblock {\em The symmetric group}, volume 203 of {\em Graduate Texts in
  Mathematics}.
\newblock Springer-Verlag, New York, second edition, 2001.
\newblock Representations, combinatorial algorithms, and symmetric functions.

\bibitem{Sam2014}
Steven~V Sam and Andrew Snowden.
\newblock Representations of categories of {G}-maps.
\newblock {\em arXiv preprint arXiv:1410.6054}, 2014.

\bibitem{Webb2007}
Peter Webb.
\newblock An introduction to the representations and cohomology of categories.
\newblock In {\em Group representation theory}, pages 149--173. EPFL Press,
  Lausanne, 2007.

\end{thebibliography}

\end{document}